\documentclass{amsart}

\usepackage{geometry}    
\usepackage[dvipdfmx]{graphicx}
\usepackage{a4wide}
\usepackage{amssymb}
\usepackage{color}
\usepackage[all]{xy}
\usepackage{comment}
\usepackage[format=plain, 
justification=raggedright,singlelinecheck=false]{caption}

\newtheorem{theorem}{Theorem}[section]    % Standard theorem environment
\newtheorem{lemma}[theorem]{Lemma}          % Lemma environment with numbering 
\newtheorem{proposition}[theorem]{Proposition}  
\newtheorem{claim}[theorem]{Claim}  
\newtheorem{example-definition}[theorem]{Example-Definition} 

\newtheorem{corollary}[theorem]{Corollary} 
\newtheorem{conjecture}[theorem]{Conjecture}
\theoremstyle{definition}
\newtheorem{definition}[theorem]{Definition}
\newtheorem{rem-def}[theorem]{Remark-Definition}
\newtheorem{remark}[theorem]{Remark}

\newtheorem{example}[theorem]{Example}    % Definition environment with 
\newtheorem*{remark-no-number}{Remark}             % Unnumbered environment for remarks.
\newtheorem*{definition*}{Definition}     
\newtheorem*{example*}{Example}     
\numberwithin{equation}{section}

\newcommand{\F}{\mathcal F_{ob} }
\newcommand{\cF}{\mathcal F_\xi }
\newcommand{\Int}{{\rm Int}}

\newcommand{\Q}{\mathbb{Q}}

\newcommand{\C}{\mathbb{C}}

\newcommand{\e}{\varepsilon}
\newcommand{\Diff}{{\rm Diff}^+}
\newcommand{\MCG}{{\mathcal MCG}}
\newcommand{\ri}{{\sf right}}
\newcommand{\A}{\mathcal{A}_{C}(S;P)}
\newcommand{\vphi}{\varphi}

\newcommand{\cL}{{\bf L}}
\title[quasi-right-veering braids]{Quasi-right-veering braids and non-loose links}
\author{Tetsuya Ito}
\address{Department of Mathematics, Kyoto University, Kyoto 606-8502, JAPAN}
\email{tetitoh@math.kyoto-u.ac.jp}
\author{Keiko Kawamuro}
\address{Department of Mathematics,   
The University of Iowa, Iowa City, IA 52242, USA}
\email{keiko-kawamuro@uiowa.edu}
\date{\today} 

\begin{document}

\begin{abstract}
We introduce a notion of ``quasi-right-veering'' for closed braids, which plays an analogous role to ``right-veering'' for open books.  
We show that a transverse link $K$ in a contact 3-manifold $(M,\xi)$ is non-loose if and only if every braid representative of $K$ with respect to every open book decomposition that supports $(M,\xi)$ is quasi-right-veering. 
We also show that several definitions of ``right-veering'' closed braids are equivalent.
\end{abstract}

\maketitle

\tableofcontents 

\section{Introduction}

The dichotomy between tight and overtwisted is fundamental to $3$-dimensional contact topology and detecting tightness of a given contact structure often arises as an important problem.

A Legendrian or transverse link in a contact 3-manifold is called \emph{loose} if the complement of the link contains an overtwisted disk, and otherwise it is called \emph{non-loose}.

In the classification of Legendrian and transverse links in contact 3-manifolds, non-loose vs. loose dichotomy plays a similar role to the tight vs. overtwisted dichotomy in the classification of contact structures. 
For instance, overtwisted contact structures are classified by homotopy equivalence \cite{el}, on the other hand loose null-homologous Legendrian (resp. transverse) links are coarsely classified  by classical invariants called the Thurston-Bennequin number and the rotation number (resp.  the self-linking number) \cite{ef, et}. Here, `coarse' means up to contactomorphism, {\em smoothly} isotopic to the identity. 

%With the Giroux correspondence \cite{gi} between contact $3$-manifolds and open books, Honda, Kazez and Mati\'c \cite{hkm,hkm2} show that checking right-veeringness of a mapping class gives an effective way to detect tightness of the compatible contact structure.

Recall a result of Honda, Kazez and Mati\'c.
 
\begin{theorem}\cite[Theorem 1.1]{hkm}
\label{cor:HKM}
A contact 3-manifold $(M,\xi)$ is tight if and only if for every  open book decomposition $(S,\phi)$ of $(M,\xi)$, $\phi$ is right-veering.
\end{theorem}

In \cite{hkm}, Honda, Kazez and Mati\'c  also define the \emph{fractional Dehn twist coefficient} (\emph{FDTC}).
The FDTC is an invariant of an open book decomposition and detects right-veering-ness of the monodromy. Hence the FDTC can be used to determine tight or overtwisted of the compatible contact structure \cite{ch,hkm2,ik4}.

As a natural counterpart of right-veering mapping classes,  {\em right-veering closed braids} (with respect to  general open books) have been defined and studied in the literature \cite{bg,bvv,pl}. 
In Section \ref{sec:FDTC} we define a closed braid $\cL$ in open 
book $(S,\phi)$ and discuss how to assign an element $[\varphi_L]$ of the mapping class group  for  $\cL$. Then 
as a counterpart of the FDTC $c(\phi, C)$, we define $c(\phi, \cL, C)$ the {\em FDTC for a closed braid $\cL$ with respect to an open book $(S,\phi)$} and a boundary component $C$ of $S$. 
The definitions given in this paper are more rigorous than those in \cite{ik2}. 

\begin{center}
\begin{tabular}{|c|c|} \hline
open book $(S, \phi)$ & closed braid $\cL$ w.r.t. $(S, \phi)$ \\ \hline
right-veering mapping class & (quasi-)right-veering closed braid \\
$\phi \in \MCG(S)$ & $[\vphi_L] \in \MCG(S, P)$ \\ \hline
FDCT of $\phi$ w.r.t. $C$ & FDTC of braid $\cL$ w.r.t. $C$ \\
$c(\phi, C)$ &   $c(\phi, \cL, C):=c([\vphi_L], C)$ \\ \hline
\end{tabular}
\end{center}

In \cite{ik2} we see that various results on open books and the FDTC are translated to results on closed braids and the FDTC for closed braids. 
This gives us some hope that open books and closed braids in open books can be treated in a unified manner.

However, this is too optimistic: 
Note that any non-right-veering open book supports an overtwisted contact structure \cite{hkm}, but not every 
non-right-veering closed braid is loose. 
A simple example of this fact is a non-right-veering closed braid in an open book decomposition of a tight contact 3-manifold.

In this paper, we find a condition on closed braids to be loose. In Definition~\ref{defn:qveer} we introduce \emph{quasi-right-veering} closed braids.  
After studying basic properties of quasi-right-veering braids we show that it is the quasi-right-veering condition on closed braids that plays the same role as the right-veering condition on open books in Theorem~\ref{cor:HKM}. 
Our first main result is the following: \\

\noindent
{\bf Theorem~\ref{theorem:main}.}
{\em A transverse link $\mathcal{T}$ in a contact 3-manifold $(M,\xi)$ is non-loose if and only if every braid representative of $\mathcal{T}$ with respect to every open book decomposition  of $(M,\xi)$ is quasi-right-veering.  }\\

In Theorem~\ref{theorem:main} we allow the transverse link $\mathcal T$ to be empty. 
Our definition of quasi-right-veering implies that the empty braid with respect to an open book $(S,\phi)$ is quasi-right-veering 
if and only if $\phi$ is right-veering.
Having a loose empty link can be interpreted as having an overtwisted underlying contact structure.  
Therefore, Theorem~\ref{cor:HKM} follows as a corollary of Theorem~\ref{theorem:main}.

In Sections \ref{sec:depth} and \ref{sec:non-loose} we present more results concerning non-loose links.

The invariant {\em depth} is a measurement of non-looseness introduced by Baker and Onaran \cite{bo}. 
In Theorem~\ref{theorem:depthone} we relate depth-one links and non-quasi-right-veering braids. \\

%\noindent
%{\bf Theorem~\ref{theorem:depthone}.}
%{\em
%Let $(S, \vphi)$ be an abstract open book supporting $(M, \xi)$. 
%Let $B$ denote the binding of the induced open book decomposition  of $M$ and $L$ be a closed braid with respect to $(S, \vphi)$. 
%Let $K := B \cup L$ which is a transverse link in $(M, \xi)$.
%The depth $d(K)=1$ if and only if the braid $\cL =[(S, \vphi), L)]$ is not quasi-right-veering.   \\ }

\noindent
{\bf Theorem~\ref{theorem:depthone}.}
{\em
Let $(S, \phi)$ be an open book decomposition of $(M, \xi)$ and let $\cL$ be a closed braid in open book $(S,\phi)$.
The depth of axis-augmented transverse link for $\cL$ is one if and only if the braid $\cL$ is not quasi-right-veering.  \\ }

Here the \emph{axis-augmented transverse link for $\cL$} is a transverse link represented as the union of $\cL$ and the binding of the open book $(S,\phi)$ (see Section \ref{sec:depth} for precise definition).

Theorem~\ref{theorem:nonloose} below is a result on braids and it can be seen as a generalization of \cite[Corollary 1.2]{ik4} a result on open books. \\

\noindent{\bf Theorem~\ref{theorem:nonloose}.} 
{\em Let $(S,\phi)$ be a planar open book decomposition of a contact 3-manifold $(M,\xi)$. If a transverse link $\mathcal{T} \subset (M,\xi)$ is represented by a closed braid $\cL%=[(S,\vphi),L]
$ such that $c(\phi,\cL,C)>1$ for every boundary component $C$ of $S$, then $\mathcal{T}$ is non-loose.}\\

Finally in Section \ref{sec:comparison} we address one subtle but important issue on right-veering closed braids. 
As mentioned above, a couple of different definitions of right-veering closed braids have been existing in the literature (cf. \cite{bg, bvv, pl}), which we call $\partial$-$(\partial+P)$, $\partial$-$\partial$, and $\partial$-$P$ right-veering (see Definition~\ref{def:dd-right-veering}). 
We show that they are essentially equivalent though there are subtle differences as stated in Theorem~\ref{theorem:equivalencerv}.\\

\noindent{\bf Corollary~\ref{cor:rveer}.}
{\em For $\psi \in \MCG(S,P)$ the following are equivalent.
\begin{enumerate}
\item $\psi$ is $\partial$-$(\partial+P)$ right-veering.
\item $\psi$ is $\partial$-$\partial$ right-veering.
\item $\psi$ is $\partial$-$P$ right-veering.
\end{enumerate}
}
 
%%---------------- SECTION 2 ----------------

\section{Closed braids as mapping classes and their FDTC}
\label{sec:FDTC}

In this section we review the distinguished monodromy for a closed braid. 
The distinguished monodromy is an element of the mapping class group of a surface with marked points. 
We also review the definition of FDTC for closed braids and prove its well-definedness.

Let $S \simeq S_{g,d}$ be an oriented compact surface with genus $g$ and $d$ boundary components.  Throughout the paper we assume $d>0$. 
Let $P = \{p_1,\ldots,p_n\}$ be a (possibly empty) set of $n$ distinct interior points of $S$. 
Let $\MCG(S,P)$ (denoted by $\MCG(S)$ if $P$ is empty) be the mapping class group of the punctured surface $S \setminus P$, which is the group of isotopy classes of orientation preserving homeomorphisms of the surface $S$ fixing $\partial S$ pointwise and fixing $P$ setwise. 
Let $\Diff(S, \partial S)$ denote the group of orientation-preserving diffeomorphisms of $S$ that fix $\partial S$ pointwise. 
Let $\Diff(S, P, \partial S)$ be the group of orientation preserving diffeomorphisms of $S$ that fix $P$ setwise and $\partial S$ pointwise.  

Let $\phi\in\MCG(S)$ be a mapping class and $\vphi\in\Diff(S, \partial S)$ be a diffeomorphism representing $\phi$. 
We call the pair $(S,\varphi)$ \emph{abstract open book}, whereas we call the pair $(S,\phi)$ \emph{open book}.

\subsection{Generalized Birman exact sequence}

\begin{definition}
Let $P'=\{p'_1,\ldots,p'_n\}$ be a finite set of interior points of $S$. (We do not require $P=P'$.) 
Let $\{x_1, \cdots, x_n\}$ be an abstract set of $n$ points. 
A \emph{geometric $n$-braid} of $S$ joining $P \times \{0\}$ and $P' \times \{1\}$ is embedding of $n$ copies of the interval $[0,1]$ into $S \times [0,1]$,
\[ 
\begin{array}{cccccc}
\beta:& \overbrace{[0,1]\sqcup \cdots \sqcup [0,1]}^{n}&\cong& \{x_1, \cdots, x_n\}\times[0,1]& \rightarrow & S \times [0,1]\\
& & &\rotatebox{90}{$\in$}& &\rotatebox{90}{$\in$}\\
& & &(x_i,t) & \mapsto & ({\beta}^i(t),t)
\end{array}\]
such that 
$\{{\beta}^1(0),\ldots, {\beta}^n(0)\} = P$ and 
$\{{\beta}^1(1),\ldots, {\beta}^n(1)\}= P'$ as (unordered) sets. 
\end{definition}

We view the geometric braid $\beta$ as an isotopy $\{\beta_t: P \rightarrow S\ | \ t\in[0,1]\}$ of   %unordered $n$ distinct points $P$. 
the set of points $P$ such that $\beta_0 = id_P$ and $\beta_1(P)=P'$. 
We extend the isotopy $\{\beta_t\}$ to an ambient isotopy $\{\widehat\beta_{t}:S \rightarrow S: {\rm diffeomorphism} \ | \ t\in[0,1]\}$ of $S$ so that  
\begin{itemize}
\item
$\widehat\beta_t |_{P}=\beta_t$.
\item $\widehat\beta_t |_{\partial S} =id$ for all $t\in[0,1] $.
\item
$\widehat\beta_0 = id_{(S, P)}$.
\item 
$\widehat\beta_t$ is isotopic to $id_S$ for all $t\in[0,1]$ if we forget the marked points in $S$. 
\end{itemize}
This extension is unique up to isotopy of $S$ fixing $\partial S$ and $\beta_t(P)$. 
We obtain a diffeomorphism
\begin{equation}\label{eq:beta-hat}
\widehat\beta_1: (S,P) \rightarrow (S,P')
\end{equation} 
and call $\widehat\beta_1$ a 
\emph{diffeomorphism associated to the geometric braid} $\beta$.

When $P=P'$, the set of isotopy classes of geometric $n$-braids forms a group. 
Regardless of the choice of $P$, the group is isomorphic to $\pi_1(C(S, n))$ the fundamental group of the configuration space of $n$ distinct, unordered points in $S$. 
In this paper we denote $\pi_1(C(S, n))$ by $B_{n}(S)$ and call it the {\em $n$-stranded surface braid group} of $S$.  
We denote by $[\beta]$ the element of $B_n(S)$ represented by the geometric braid $\beta$ in $S\times [0,1]$ joining $P\times \{0\}$ and $P\times\{1\}$. 

When $P=P'$, the above diffeomorphism $\widehat{\beta}_1$ yields a well-defined homomorphism $i$ which we call the \emph{push map}:
\[
\begin{array}{ccc}
i: B_n(S) &\longrightarrow & \MCG(S,P)\\
\rotatebox{90}{$\in$} & &\rotatebox{90}{$\in$}\\
{}[\beta]  & \longmapsto & [\widehat{\beta}_1] 
\end{array}
\]

Suppose that $\vphi\in\Diff(S, P, \partial S)$. 
Forgetting the points $P$, the diffeomorphism $\varphi:(S,P)\rightarrow (S,P)$ can be regarded as a diffeomorphism $\varphi:S\to S$. 
This defines a surjective homomorphism $f:\MCG(S,P) \rightarrow \MCG(S)$ called the \emph{forgetful map}.

The push map $i$ and the forgetful map $f$ give the generalized Birman exact sequence \cite[Theorem 9.1]{FM}
\begin{equation}
\label{eqn:Birman}
1 \rightarrow B_n(S)  \stackrel{i}{\rightarrow} \MCG(S,P)  \stackrel{f}{\rightarrow} \MCG(S) \rightarrow 1.
\end{equation}

\subsection{The distinguished monodromy}\label{subsec:2.2}

Let $\phi \in \MCG(S)$ be a mapping class. 
Take a collar neighborhood $\nu(\partial S)$ of the boundary $\partial S$. 
Choose $\varphi \in \Diff(S, \partial S)$ that represents $\phi\in \MCG(S)$ and satisfies:
\begin{equation}
\label{eqn:collar}
\varphi|_{\nu (\partial S)}=id.
\end{equation} 

Let 
$$M_{(S, \varphi)} := S \times[0,1] \slash \sim,$$
where ``$\sim$'' denotes the equivalence relation 
\[ (x,1)\sim(\varphi(x),0) \ \ \mbox{for all } x \in S, \quad (x,t)\sim (x,s) \ \ \mbox{for all } x \in \partial S \mbox{ and }  t,s \in [0,1].\] 
Denote the quotient map by $\pi: S\times[0,1] \rightarrow M_{(S,\varphi)}$. 

The manifold $M_{(S,\varphi)}$ is naturally equipped with the following open book decomposition: 
For $t \in [0,1]$, we call $S_t:= \pi (S\times\{t\}) \subset M_{(S,\varphi)}$ a \emph{page} and $B:=\pi (\partial S \times\{t\}) \subset M_{(S,\varphi)}$ (this does not depend on the choice of $t$) the \emph{binding}.
The quotient map $\pi$ restricted on $S\times\{t\}$ gives a diffeomorphism $\pi|_{S \times\{t\}}: S\times\{t\} \rightarrow S_t$.
Composing $(\pi|_{S \times\{t\}})^{-1}$ and the projection $pr: S\times[0,1] \ni (x,t) \mapsto x\in S$, every page $S_t$ is diffeomorphic to $S$. 
\[
\begin{array}{ccccc}
 S_t &  \stackrel{
(\pi|_{ S\times\{t\}})^{-1}}{\longrightarrow}  & S \times \{t\} & \stackrel{pr}{\longrightarrow} & S \\
\rotatebox{90}{$\in$}& &\rotatebox{90}{$\in$} & &\rotatebox{90}{$\in$} \\
\pi(x,t) & \longmapsto & (x, t) & \mapsto & x 
\end{array}
\]
Denote the diffeomorphism by $p_t := pr \circ 
(\pi|_{ S\times\{t\}})^{-1}:S_t\to S$. 
We may extend $p_t$ to 
\begin{equation}\label{eq:projection-p}
p: M_{(S,\varphi)} \to S
\end{equation}
by setting $p|_{S_t}=p_t$ and call the map $p$ a {\em projection. }

\begin{definition}[Closed braids w.r.t. $(S, \vphi)$]
\label{def:braid-vphi} 
A \emph{closed $n$-braid $L$ with respect to an abstract open book $(S,\varphi)$} is an oriented link in the 3-manifold $M_{(S,\varphi)}$ such that $L \subset M_{(S,\varphi)} \setminus B$ and $L$   intersects every page $S_t$ positively and transversely at $n$ points. 
\end{definition}

Suppose that under the projection $p:M_{(S,\varphi)}\to S$ the $n$ intersection points $L \cap S_0$ are contained in $\nu(\partial S)$ :
\begin{equation}
\label{eqn:P}
P:=p(L\cap S_0) \subset \nu(\partial S).
\end{equation} 
Recall that in the manifold $M_{(S,\varphi)}$, a point $x$ in the page $S_1$ is identified with the point $\varphi(x)$ in $S_0$.
By (\ref{eqn:collar}), we note that $p(L\cap S_1)=\varphi\circ p(L\cap S_1) = p(L\cap S_0)=P$. 
Cutting the manifold $M_{(S,\varphi)}$ along the page $S_0$, the closed braid $L$ gives rise to a geometric $n$-braid denoted by $\beta_{L} \subset S \times [0,1]$ joining $P\times\{0\}$ and $P\times\{1\}$ such that $%\tau\circ
\pi(\beta_L)=L$. 
% \marginpar{\tiny If we don't assume $P\subset \nu(\partial S)$, $\beta_L$ should be denoted by $\beta_{L, \varphi}$}

By (\ref{eqn:collar}) and (\ref{eqn:P}) we have $\varphi|_P = id$ hence we may view $\varphi$ as an element of $\Diff(S, P, \partial S)$. % the group of orientation preserving diffeomorphisms of $S$ that fixes $P$ setwise and $\partial S$ pointwise. 
In order to distinguish $\varphi$ in $\Diff(S, \partial S)$ and $\varphi$ in $\Diff(S, P, \partial S)$, we denote the latter by $j(\varphi)$.

%Note that if $P \not\subset\nu(\partial S)$ then $\varphi(P) \neq P$ in general, and the homomorphism $j$ cannot be defined. 
%The map $j$ induces a homomorphism: 
%\marginpar{\tiny $j_*$ is not well-defined because $\phi$ can be represented by $\varphi$ that doesn't satisfy (2.3). Shall we delete $j_*$? -- Oh, I understandthe problem: the map $j$ is also not defined in general} 
%$$
%\begin{array}{ccc}
%j_{*}: \MCG(S) &\longrightarrow & \MCG(S,P)\\
%\rotatebox{90}{$\in$}& &\rotatebox{90}{$\in$}\\
%\phi =[\varphi] & \longmapsto & [j(\varphi)] 
%\end{array}
%$$ 

Let 
\begin{equation}\label{def of varphi_L}
\varphi_L:= \widehat{\beta_L}_1\circ j(\varphi)\in\Diff(S, P, \partial S).
\end{equation} 
Then we have 
\begin{equation}\label{eqML}
(M_{(S,\varphi)}, L) \simeq \left( (S, P) \times [0,1] \right) / \sim_{\varphi_L}
\end{equation}
where the equivalence relation ``$\sim_{\varphi_L}$'' satisfies  
$(x, 1) \sim (\varphi_L(x), 0)$ for all $x \in S$ and $(x, 1) \sim (x, t)$ for all $x\in\partial S$ and $t\in[0,1]$.

\begin{definition}[Distinguished monodromy]\label{def:distinguished-monodromy}
Let $L$ be a closed braid with respect to the abstract open book $(S,\varphi)$ satisfying the conditions (\ref{eqn:collar}) and (\ref{eqn:P}).  
The \emph{distinguished monodromy} is the mapping class
\[ [\varphi_L] \in \MCG(S,P). \] 
\end{definition}

In \cite{ik2} the above $[\varphi_L]$ is denoted by $\phi_L$. In this paper, we do not use the notation $\phi_L$ because 
we want to distinguish closed braids with respect to $(S, \vphi)$ (Definition~\ref{def:braid-vphi}) and closed braids with respect to $(S, \phi)$ (Definition~\ref{def:braids-phi}), and $L$ is only defined with respect to $(S, \vphi)$.

\begin{definition}[Braid isotopy]
We say that two closed braids $L$ and $L'$ with respect to the same abstract open book $(S,\varphi)$ are \emph{braid isotopic} if they are isotopic through closed braids.
\end{definition}

In practice, we tend to identify two closed braids if they are braid isotopic. 
In Proposition~\ref{prop:isotopic-braids-1} we discuss how a braid isotopy affects the distinguished monodromy. 

\begin{definition}[Point-changing isomorphism]\label{def:point-changing}
Recall that when $|P|=|P'|$, the groups $\MCG(S,P)$ and $\MCG(S,P')$ are isomorphic. 
We say that an isomorphism $\Theta:\MCG(S,P) \rightarrow \MCG(S,P') $ is \emph{point-changing} if $\Theta$ is defined by $\Theta([\psi])= [\theta^{-1} \circ \psi \circ \theta] $ for some orientation-preserving diffeomorphism $\theta: (S, P')\to (S, P)$ such that $\theta|_{\partial S}=id$ and $\theta$ is isotopic to $id_S$ if we forget the marked points of $P$ and $P'$.

If $P=P'$, every point-changing isomorphism is an inner automorphism of $\MCG(S,P)$. 

\end{definition}

\begin{proposition}\label{prop:isotopic-braids-1}
Let $L$ and $L'$ closed $n$-braids with respect to an abstract open book $(S, \varphi)$ satisfying (\ref{eqn:collar}).
Suppose that $P := p(L \cap S_0) \subset \nu(\partial S)$ and $P' := p(L' \cap S_{0}) \subset \nu(\partial S)$. 
If $L$ and $L'$ are braid isotopic then 
there exists a point-changing isomorphism
$$\gamma^*:\MCG(S,P)\to \MCG(S, P')$$ such that 
$[\varphi_{L'}] = \gamma^*([\varphi_L])$.
%$\phi_{L'} = \gamma^*(\phi_L)$. 
\end{proposition}

\begin{proof}
Cutting the manifold $M_{(S,\varphi)}$ along the page $S_{0}$, the closed braids $L$ and $L'$ give rise to  geometric $n$-braids $\beta:=\beta_L$ and $\beta':=\beta_{L'} \subset S\times [0,1]$,  respectively. 
Since $L$ and $L'$ are braid isotopic we have 
\begin{equation}\label{eq:beta}
[\beta'] = [\gamma^{-1} \bullet \beta \bullet  \gamma^{\varphi}] \ \mbox{ (read from the right to left)}
\end{equation}
for some geometric $n$-braid $\gamma \subset S\times [0,1]$ connecting $P' \times \{0\}$ and $P \times \{1\}$ and specified by an isotopy 
$\{\gamma_t: P' \to S \ | \ t\in [0,1]\}$
such that $\gamma_0=id_{P'}$ and $\gamma_1(P')=P$. 
%$$
%\begin{array}{cccc}
%\gamma: & \{x_1,\cdots,x_n\} \times  [0,1] & \longrightarrow & S\times[0,1] \\
%& \rotatebox{90}{$\in$}& &\rotatebox{90}{$\in$}\\
%&(x_i, t) & \longmapsto &\{ \gamma^1(t),\ldots, \gamma^n(t)\} \times \{t\}
%\end{array}
%$$ 
%connecting $P' \times \{0\}$ and $P \times \{1\}$.
Here
\begin{itemize}
\item the bullet ``$\bullet$'' is the concatenation of geometric braids.
\item $\gamma^{-1}$ is the geometric $n$-braid joining $P\times\{0\}$ and $P'\times\{1\}$ defined by 
$$(\gamma^{-1})_t :=\gamma_{1-t}.$$
\item  $\gamma^{\varphi}$ is the geometric $n$-braid joining $P'\times\{0\}$ and $P\times\{1\}$  defined by $$(\gamma^{\varphi})_t := \varphi \circ \gamma_t.$$
\end{itemize}

As done in (\ref{eq:beta-hat}), we extend $\{\gamma_t\}$ to an isotopy $\{\widehat{\gamma}_{t}:S \rightarrow S: {\rm diffeomorphism} \  | \ t\in [0,1] \}$. 
The diffeomorphism $\widehat \gamma_1: (S, P')\to(S, P)$ associated to the geometric braid $\gamma$ gives rise to a point-changing isomorphism 
\begin{equation}
\label{eq:gamma}
\begin{array}{cccc}
\gamma^*: & \MCG(S,P) & \longrightarrow & \MCG(S,P') \\
& \rotatebox{90}{$\in$}& &\rotatebox{90}{$\in$}\\
& [\psi] & \mapsto & [{\widehat{\gamma}_1}^{-1} \circ \psi \circ \widehat{\gamma}_1]
\end{array}
\end{equation}

Similarly, we extend $\{(\gamma^{\varphi})_t\}$ to an isotopy
$\{\widehat{\gamma^{\varphi}}_t: S \rightarrow S\}$. Let $i': B_n(S)\to \MCG(S, P')$ be the push map in the generalized Birman exact sequence (\ref{eqn:Birman}) where $P$ is replaced with $P'$.
Then by the definition of push map we have 
\begin{equation}
\label{eq:beta2} i'([\gamma^{-1} \bullet \beta \bullet  \gamma^{\varphi}]) = 
 [{\widehat{\gamma}_1}^{-1} \circ \widehat{\beta}_1 \circ  \widehat{\gamma^{\varphi}}_1] 
\end{equation}

%Let $j':\Diff(S, \partial S) \to \Diff(S, P', \partial S)$ and $j'_{*}: \MCG(S) \to\MCG(S,P')$ denote the homomorphisms constructed in the same way as $j$ and $j_*$ replacing $P$ with $P'$. 
Let $j'(\varphi)$ denote the diffeomorphism $\varphi$, viewed as an element of $\Diff(S, P', \partial S)$. 
By the definition of the braid $\gamma^{\varphi}$ we have
\begin{equation}
\label{eq:beta3}
\widehat{\gamma^\varphi}_1 = \varphi \circ\widehat{\gamma}_1 \circ \varphi^{-1}=j(\varphi) \circ\widehat{\gamma}_1 \circ j'(\varphi)^{-1}
\end{equation}

The following calculation concludes the proposition: 
\begin{align*}\label{eq:conjugation}
[\varphi_{L'}]
%&= i'([\beta']) j'_*(\phi)  & (\textrm{by Definition } \ref{def:distinguished-monodromy}) \\
&= [\widehat{\beta'}_1 \circ j'(\varphi)] =i'([\beta'])[j'(\varphi)] 
 & (\textrm{by Definition } \ref{def:distinguished-monodromy}) \\
&= i'([\gamma^{-1} \bullet \beta \bullet \gamma^{\varphi}])  [j'(\varphi)] & (\textrm{by } (\ref{eq:beta}) )\\
&= [{\widehat\gamma_1}^{-1} \circ \widehat{\beta}_1 \circ \widehat{\gamma^{\varphi}}_1 \circ j'(\varphi)]& (\textrm{by } (\ref{eq:beta2}) )\\
&= [{\widehat\gamma_1}^{-1} \circ \widehat{\beta}_1 \circ j(\varphi) \circ \widehat\gamma_1 \circ j'(\varphi)^{-1} \circ j'(\varphi)]& (\textrm{by } (\ref{eq:beta3}) )\\
&= [{\widehat\gamma_1}^{-1} \circ ( \widehat\beta_1 \circ j(\varphi)) \circ \widehat\gamma_1 ] \\
&= \gamma^* ([\varphi_L]) & (\textrm{by Definition } \ref{def:distinguished-monodromy} \textrm{ and } (\ref{eq:gamma}))\\
\end{align*} 
\end{proof}

%Let us take this opportunity to explain and give a strict definition for a notion of \emph{a closed braid in an open book $(S,\phi)$} which we used in our previous paper. We noticed that this notation is confusing and need more care.

In our previous papers such as \cite{ik2}, by ``a closed braid with respect to the open book $(S,\phi)$'' we mean a closed braid with respect to some  \emph{abstract} open book $(S,\varphi)$ with $[\varphi]=\phi$.
As long as we have geometric argument (eg. open book foliations),  this causes no trouble. However, when we discuss connection to mapping class groups (eg. the distinguished monodromy), we need understand what happens if we take another diffeomorphism $\varphi'$ with $[\vphi']=\phi$.

Assume that $\vphi$ and $\vphi'$ are isotopic.
Let $\{\rho_t\in\Diff(S, \partial S) \ |\ t\in[0,1]\}$ be an isotopy between $id_S$ and $\vphi' \circ \vphi^{-1}$, and define 
$$\overline{\rho}:S\times[0,1] \rightarrow S\times[0,1]$$
by $$\overline{\rho}(x,t)=(\rho_{1-t}(x),t).$$ 
Since $\overline{\rho}(x,1)=(x,1)$ and $\overline{\rho}(\vphi(x),0)=(\vphi'(x),0)$, $\overline{\rho}$ induces a diffeomorphism 
\begin{equation}\label{eq:diffeo-rho}
\rho: M_{(S,\varphi)} \rightarrow M_{(S,\varphi')}
\end{equation}
preserving the pages: For each $t \in [0,1]$ the page $S_t$ of $M_{(S,\varphi)}$ is mapped to the page $S_t$ of $M_{(S,\varphi')}$.
In particular, if $L$ is a closed braid with respect to $(S,\varphi)$ then $\rho(L)$ is a closed braid with respect to $(S,\vphi')$.

\begin{definition}[Closed braids w.r.t. $(S,\phi)$]
\label{def:braids-phi}
Let $L \subset M_{(S,\varphi)}$ (resp. $L' \subset M_{(S,\varphi')}$) be a closed braid with respect to $(S, \vphi)$ (resp. $(S,\varphi')$).  
We say that triples $((S,\varphi),L)$ and $((S,\varphi'),L')$ are \emph{equivalent}, if 
$[\vphi]=[\vphi']$ and the closed braid $\rho(L) \subset M_{(S,\varphi')}$ is braid isotopic to $L'$. 

In particular, closed braids $L$ and $L'$ are called {\em braid isotopic} if 
 $\vphi=\vphi'$ and $((S,\varphi),L)$ and $((S,\varphi'),L')$ are equivalent.

The equivalence class of $((S,\varphi),L)$ with $[\varphi]=\phi$ is called a \emph{closed braid with respect to the open book $(S,\phi)$} and denoted by $[((S,\varphi),L)]$ or simply $\cL$. 
\end{definition}

The distinguished monodromy $[\varphi_L]$ is an invariant of the equivalence class of $((S,\varphi),L)$ (i.e., a closed braid with respect to $(S, \phi)$) up to point-changing isomorphism.

\begin{theorem}
\label{thm:well-defined-up-to-point}
Let $(S,\varphi)$ and $(S,\varphi')$ be abstract open books with $\vphi=\vphi'=id$ on $\nu(\partial S)$ and $[\varphi]=[\varphi'] \in \MCG(S)$.
Let $L$ and $L'$ be closed braids with respect to $(S,\varphi)$ and $(S,\vphi')$, respectively. 
Let $P:=p(L\cap S_0)$ and $P':=p(L' \cap S_0)$. Suppose that the conditions (\ref{eqn:collar}) and (\ref{eqn:P}) are satisfied for both braids. 
If $((S,\varphi),L)$ and $((S,\varphi'),L')$ are equivalent then there is a point-changing isomorphism $\gamma^{*}: \MCG(S,P) \rightarrow \MCG(S,P')$ such that $[\vphi'_{L'}]=\gamma^{*}([\vphi_{L}])$.
\end{theorem}

\begin{proof}

Since $\vphi=\vphi'=id$ on $\nu(\partial S)$ 
we may assume that the above isotopy $\{\rho_t:S\rightarrow S \ |\ t\in[0,1]\}$ between $id_S$ and $\vphi' \circ \vphi^{-1}$
satisfies $\rho_t=id$ on $\nu(\partial S)$ for all $t\in[0,1]$. 
Define $\rho$ and $\overline\rho$ as above. 
Then $\rho=id$ on $\partial S \times [0,1]$ 
and 
$p(L\cap S_0)= p(\rho(L) \cap S_0)=:P \subset \nu(\partial S).$ 
This implies that 
\begin{equation}\label{eq:j(vphi)}
[j(\vphi)] = [j(\vphi')] \in \MCG(S, P).
\end{equation}

%Then $\rho=id$ on $\partial S \times [0,1]$ so the geometric braids $\beta_L$ and $\overline{\rho}(\beta_L)$ have the same endpoints. 

Let $\beta_L$ (resp. $\beta_{\rho(L)}$)  denote the geometric braid obtained from $L$ (resp. $\rho(L)$) by cutting the manifold $M_{(S,\varphi)}$ (resp. $M_{(S,\varphi')}$) along the page $S_{0}$.  
Note that $\beta_{\rho(L)} = \overline\rho(\beta_L)$.

For $(x, t)\in S\times[0,1]$ and $s\in[0,1]$ let $$\overline{\rho}_{s}(x,t):=(\rho_{s(1-t)}(x),t).$$ 
Then $\{ \overline\rho_s \in \Diff(S\times[0,1])  \ | \ s\in[0,1] \}$ 
gives an isotopy between $\overline{\rho}_0=id_{S\times[0,1]}$ and $\overline{\rho}_1=\overline{\rho}$. 
Therefore, the geometric braids $\beta_L$ and $\overline{\rho}(\beta_L)$ are isotopic through the family of geometric braids $\{\overline\rho_s(\beta_L) \ | \ s \in[0,1]\}$ having the same endpoints, and in the surface braid group  
\begin{equation}\label{eq:beta_L}
[\beta_{\rho(L)}]=[\overline{\rho} (\beta_L)]=[\beta_L] \in B_{n}(S).
\end{equation}

By (\ref{eq:j(vphi)}) and (\ref{eq:beta_L}) we have an identity between the distinguished monodromies of $L$ and $\rho(L)$: 
$$[\varphi'_{\rho(L)}]=i([\beta_{\rho(L)}])[j(\varphi')]=i([\beta_L])[j(\varphi)]=[\varphi_L] \in \MCG(S,P).$$ 
Since $\rho(L)$ and $L'$ are braid isotopic, by Proposition~\ref{prop:isotopic-braids-1}, $[\varphi'_{\rho(L)}]$ and $[\varphi'_{L'}]$ are related by a point-changing isomorphism. 
\end{proof}

%If $((S,\varphi),L)$ and $((S,\varphi'),L')$ are equivalent, Proposition~\ref{prop:isotopic-braids-1} states that there is a point-changing isomorphism $\gamma^{*}: \MCG(S,P) \rightarrow \MCG(S,P')$ such that $[\vphi_{L'}]=\gamma^{*}([\vphi_{L}])$.
%$\phi_{L'}=\gamma^{*}(\phi_{L})$.

%In particular, as we will discuss below, a notion of FDTC, for (quasi)-right-veering closed braids in an open book $(S,\phi)$ will be well-defined.

\begin{remark}
\label{rem:def-phiL}
One can develop a notion of the distinguished monodromy in more general setting.
In the definition of the distinguished monodromy, one can replace the conditions (\ref{eqn:collar}) and (\ref{eqn:P}) with a weaker condition $\varphi(P)=P$ since the latter condition enables us to define a diffeomorphism $j(\varphi) \in \Diff(S, P, \partial S)$. Similarly, Theorem \ref{thm:well-defined-up-to-point} is valid under similar weaker conditions like $\varphi(P)=P$ and $\varphi'(P')=P'$
 where $P:=p(L\cap S_0)$ and $P':=p(L' \cap S_0)$ (in general case, $[\varphi_L]$ and $[\varphi'_{\rho(L)}]$ are not identical, but related by a point-changing isomorphism).

However, for a given $\varphi \in \Diff(S,\partial S)$, finding an $n$-point set $P\subset S$ with $\varphi(P)=P$ is a difficult task since this amounts to find a periodic orbit of $\varphi$ explicitly. This is a reason why the assumptions (\ref{eqn:collar}) and (\ref{eqn:P}) make it easier to handle the distinguished monodromy of a closed braid.

\end{remark}

\subsection{The fractional Dehn twist coefficient for closed braids} 
\label{section:FDTC}

Let $C$ be a boundary component of $S$. 
Denote the \emph{fractional Dehn twist coefficient} (\emph{FDTC}) with respect to $C$, which is defined by Honda, Kazez and Mati\'c \cite{hkm}, by $$c(-,C):\MCG(S,P)\rightarrow \Q.$$

%{\color{red} \marginpar{\tiny where do we move this sentence?} We have $c(\psi[\varphi_{L}]\psi^{-1}, C) = c([\varphi_{L}], C)$ for any $\psi\in\MCG(S,P)$. }
%In fact a stronger result holds. 

\begin{proposition}\label{prop:isotopic-braids} 
Let $(S,\varphi)$ and $(S,\varphi')$ be abstract open books with $\vphi=\vphi'=id$ on $\nu(\partial S)$ and $[\varphi]=[\varphi'] \in \MCG(S)$. 
Let $L$ and $L'$ be closed braids with respect to $(S,\varphi)$ and $(S,\vphi')$, respectively, satisfying the conditions (\ref{eqn:collar}) and (\ref{eqn:P}). 
If $((S,\varphi),L)$ and $((S,\varphi'),L')$ are equivalent then for every boundary component $C$
$$c([\vphi_L], C)=c([\vphi'_{L'}], C).$$
%
%Let $L$ and $L'$ be closed braids with respect to $(S, \varphi)$ satisfying (\ref{eqn:collar}), $P := p(L \cap S_0) \subset \nu(\partial S)$, and $P' := p(L' \cap S_{0}) \subset \nu(\partial S)$. 
%If $L$ and $L'$ are braid isotopic then 
%$$c(\phi, L , C) = c(\phi, L', C)$$ for every boundary component $C$.   
\end{proposition}

\begin{proof}
%This is a corollary of Proposition~\ref{prop:isotopic-braids-1} and Theorem~\ref{thm:well-defined-up-to-point}. 

Let $P:=p(L\cap S_0)$ and $P':=p(L' \cap S_0)$. 
By Theorem~\ref{thm:well-defined-up-to-point} 
there is a point-changing isomorphism $\gamma^{*}: \MCG(S,P) \rightarrow \MCG(S,P')$ such that $[\vphi'_{L'}]=\gamma^{*}([\vphi_{L}])$. 
By the properties of the isotopy $\{\widehat\gamma_t\}$ in the proof of Proposition~\ref{prop:isotopic-braids-1}, the following diagram commutes:
\[
\xymatrix{
\MCG(S,P) \ar[rr]^{\gamma^{*}} \ar[rd]_{c(-,C)}& & \MCG(S,P')\ar[ld]^{c(-,C)}\\
& \Q &
}
\]
%By (\ref{eq:conjugation}) we have 
Therefore 
$$c([\varphi_{L}], C) = c(\gamma^{*}([\varphi_L]), C) = c([\varphi'_{L'}], C).$$
\end{proof}

When $P'=P$ the isomorphism $\gamma^{*}$ is an inner automorphism of $\MCG(S, P)$ and the commutativity implies invariance of the FDTC under conjugation, 
$c(\psi[\varphi_{L}]\psi^{-1}, C) = c([\varphi_{L}], C)$ for any $\psi\in\MCG(S,P)$.

Now we are ready to define the FDTC for a braid. 

\begin{definition}
\label{def:FDTC-braid}
Let $L$ be a closed braid with respect to an abstract open book book $(S, \varphi)$.
Suppose that $L'$ is a closed braid with respect to $(S, \varphi)$ such that $L$ and $L'$ are braid isotopic and the conditions (\ref{eqn:collar}) and (\ref{eqn:P}) are satisfied for $((S, \vphi), L')$. 
The {\em fractional Dehn twist coefficient (FDTC) of the equivalence class $\cL = [(S, \vphi), L)]$ with respect to $C$} is the FDTC of the distinguished monodromy $[\varphi_{L'}]$ with respect to $C$ and denote it by $c(\phi, \cL, C)$ (in \cite{ik2} it is denoted by $c(\phi,L,C)$). Namely, 
\[ c(\phi,\cL,C) := c([\varphi_{L'}],C).\]
\end{definition}

Thanks to Proposition~\ref{prop:isotopic-braids}, the FDTC $c(\phi,\cL,C)$ is well-defined. 

If a braid $L$ is empty we set $P=\emptyset$ and  define the distinguished monodromy $[\varphi_{L}] := [\varphi]=\phi$.  Hence the FDTC of the empty closed braid is equal to the FDTC of the monodromy of the open book.

%In practice, when we consider $c(\phi,L,C)$ it is often convenient, as done in \cite{ik2}, to take $P=p(L\cap S_0)$ so that $P$ is contained in a collar neighborhood of a single boundary component $C$ rather than the whole $\partial S$. 

\section{Quasi-right-veering maps}
\label{sec:qveer}

%From now on, we will not pay much attention to the distinction between a diffeomorphism and its corresponding mapping class. When we say ``$L$ is a closed braid with respect to an open book $(S,\phi)$'' the reader may understand that we are choosing a representative  $\varphi \in \Diff(S)$ of $\phi\in\MCG(S)$ with (\ref{eqn:collar}), and that $L$ is a closed braid with respect to the abstract open book $(S,\varphi)$. Also, when we consider the distinguished monodromy $\phi_L$ or the FDTC $c(\phi, L, C)$ we always assume $P=p(S_0\cap L) \subset \nu(\partial S)$. 

In this section we introduce a partial ordering ``$\ll_{\sf right}$'' and quasi-right-veering closed braids, then we compare right-veering and quasi-right-veering. 
We use the same notations in the previous section. 

%Thus, when we discuss a closed braid $L$ and its distinguished monodromy $\phi_L$ there exists a particular boundary component $C$ of $S$ such that $L$ has {\color{red}\marginpar{do we need this condition? or can we weaken it to $P\subset \nu(\partial S)$} $P:=p(L \cap S_{0}) \subset \nu(C)$. }

\subsection{Strongly right-veering partial ordering  ``$\ll_{\sf right}$'' }

\begin{definition}
For each boundary component $C$ of $S$, we choose a base point $\ast_{C} \in C$. Let $\mathcal{A}_{C}(S,P)$ be the set of isotopy classes of 
properly embedded arcs $\gamma:[0,1] \rightarrow S \setminus P$ satisfying $\gamma(0) = \ast_C$ and $\gamma(1)\in \partial S \setminus \{*_C\}$.  Here, by isotopy we mean isotopy fixing the end points $\gamma(0)$ and $\gamma(1)$. 
\end{definition}

%We do not allow $\gamma\in\A$ to have $\gamma(1) \in P$ but we allow $\gamma(1) \in (C\setminus \{*_C\})$. 
For simplicity of notation, an actual arc $\gamma: [0,1] \rightarrow S$ representing its isotopy class $[\gamma] \in \mathcal{A}_{C}(S,P)$ 
may be denoted by the same symbol, $\gamma$. 
We may call an element of $\mathcal{A}_{C}(S,P)$ simply an \emph{arc} $\gamma$ instead of the isotopy class of $\gamma$.  

We say that two arcs $\alpha$ and $\beta$ intersect \emph{efficiently} if they attain the minimal geometric intersection number among all the arcs isotopic to them. 

\begin{definition}[Right-veering total ordering $\prec_\ri$]
Let $\alpha,\beta \in \mathcal{A}_{C}(S,P)$.
Suppose that (arcs representing) $\alpha$ and $\beta$ intersect efficiently. 
We denote $\alpha \prec_{\sf right} \beta$ and say that $\beta$ lies on the \emph{right side} of $\alpha$ if the arc $\beta$ lies on the  right side of $\alpha$ in a small neighborhood of the base point $\ast_{C}$. 
\end{definition}

In \cite{hkm} and \cite{ik2}, where the set $P$ is empty, the symbol `$`>$'' is used in the place of  ``$\prec_{\sf right}$''. 

The order ``$\prec_{\sf right}$'' is a total ordering. 
For any family of arcs $\{ \alpha_i \} \subset \mathcal{A}_{C}(S,P)$ we can always put them in a position simultaneously so that $\alpha_i$ and $\alpha_j$ intersect efficiently for any pairs $(i, j)$. This can be done, for example, by choosing a hyperbolic metric on $S\setminus P$ and realizing the arcs as geodesics.

We introduce another ordering ``$\ll_{\sf right}$''
which plays a central role in this paper.

\begin{definition}[Strongly right-veering partial ordering $\ll_\ri$]
For arcs $\alpha,\beta \in \mathcal{A}_{C}(S,P)$, we define $\alpha \ll_{\sf right} \beta$ if there exists a sequence of arcs $\alpha_0,\dots,\alpha_k \in \A$ such that 
\begin{equation}
\label{eqn:defn-ll}
\alpha = \alpha_0 \prec_{\sf right} \alpha_1 \prec_{\sf right}  \cdots \prec_{\sf right} \alpha_k = \beta, \    \mbox{ and } 
\end{equation}
\begin{equation}
\label{eqn:defn-ll2}
 \Int(\alpha_i) \cap \Int(\alpha_{i+1})= \emptyset \mbox{ for all } i=0,\dots, k-1.
\end{equation}
\end{definition}

In the rest of the subsection, we study properties of $\ll_\ri$. 

By the definition it is easy to see that $\ll_\ri$ is a partial ordering, i.e., $\alpha\ll_\ri \beta$ and $\beta\ll_\ri\gamma$ imply $\alpha\ll_\ri\gamma$.  
If the puncture set $P$ is empty, then \cite[Lemma 5.2]{hkm} shows that the ordering $\ll_{\sf right}$ coincides with $\prec_{\ri}$. However, when $P$ is non-empty  $\ll_\ri$ is {\em not} a total ordering and there is difference between $\prec_{\ri}$ and $\ll_{\ri}$ as shown in Proposition~\ref{lemma:b-right-p-bigon}. To see the difference we first introduce the following notion.

\begin{definition}[Boundary right $P$-bigon]
\label{defn:b-right-P-bigon}
Let $\alpha,\beta \in \mathcal{A}_{C}(S,P)$ with $\alpha \prec_{\ri} \beta$. 
Assume that there exist subarcs $\delta_{\alpha} \subset \alpha$ and $\delta_{\beta} \subset \beta$ such that 
\begin{itemize}
\item
$*_C \in \delta_\alpha \cap \delta_\beta$ 
\item
$\delta_{\alpha} \cup \delta_{\beta}$ bounds a (possibly immersed) bigon $D (\subset S)$ which lies on the right side of $\alpha$ (i.e., the orientation of $\delta_{\alpha}$, as a subarc of $\alpha$, disagrees with the orientation of $\partial D$) and 
\item
$D \cap P \neq \emptyset$ ($D$ contains some marked points).  
\end{itemize}
We call such a bigon $D$ a \emph{boundary right $P$-bigon from $\alpha$ to $\beta$}.
\end{definition}

A boundary right $P$-bigon gives an obstruction for $\alpha \ll_{\ri} \beta$: 

\begin{proposition}
\label{lemma:b-right-p-bigon}
Let $\alpha,\beta \in \mathcal{A}_{C}(S,P)$ be arcs with $\alpha \prec_{\ri} \beta$. If there is a boundary right $P$-bigon from $\alpha$ to $\beta$ then $\alpha \not \ll_{\ri} \beta$.
\end{proposition}

\begin{proof}
If there is a boundary right $P$-bigon $D$ from $\alpha$ to $\beta$ then every arc $\gamma \in \A$ that satisfies $\alpha \prec_{\ri} \gamma \prec_{\ri} \beta$ must intersect $D$ and yields either a boundary right $P$-bigon from $\alpha$ to $\gamma$ or from $\gamma$ to $\beta$ (see Figure \ref{fig:porder} (a)). 
Thus, for any sequence of arcs $\alpha=\gamma_0 \prec_{\ri} \gamma_1 \prec_{\ri} \cdots \prec_{\ri} \gamma_n = \beta$ there exists an $i \in \{0,\dots, n-1\}$ such that $\gamma_{i}$ and $\gamma_{i+1}$ forms a boundary right $P$-bigon, which means $\Int(\gamma_{i})$ and $\Int(\gamma_{i+1})$ cannot be disjoint.
\end{proof}

\begin{figure}[htbp]
\begin{center}
\includegraphics*[bb=145 612 449 716,width=90mm]{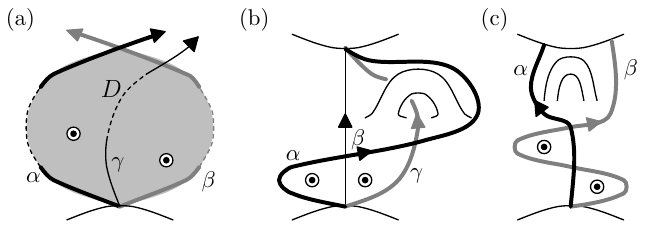}
\caption{
(a) The arc $\gamma$ with $\alpha \prec_{\ri} \gamma \prec_{\ri} \beta$ cuts the boundary right $P$-bigon $D$, yielding a boundary right $P$-bigon from $\beta$ to $\gamma$. \\
(b) $\alpha \prec_{\sf right} \beta$ and $\beta \ll_{\sf right} \gamma$, but  $\alpha \not \ll_{\sf right} \gamma$. \\
(c) $f(\alpha) \prec_{\ri} f(\beta)$ and  $\alpha \prec_{\ri} \beta$, but $\alpha \not \ll_{\ri} \beta$.}
\label{fig:porder}
\end{center}
\end{figure}

%We now clearly see the difference between $\prec_{\ri}$ and $\ll_{\ri}$. 
As a corollary, we observe that 
conditions $\alpha \prec_{\sf right} \beta$ and $\beta \ll_{\sf right} \gamma$  may \emph{not} imply $\alpha \ll_{\sf right} \gamma$ in general.  
(Also $\alpha \ll_{\sf right} \beta$ and $\beta \prec_{\sf right} \gamma$ may not imply $\alpha \ll_{\sf right} \gamma$.)
For example, the arcs depicted in Figure \ref{fig:porder} (b) satisfy $\alpha \prec_{\sf right} \beta$ and $\beta \ll_{\sf right} \gamma$ but by Proposition~\ref{lemma:b-right-p-bigon} $\alpha\not\ll_\ri\gamma$.

We conjecture the converse of Proposition~\ref{lemma:b-right-p-bigon}: 
 
\begin{conjecture}
\label{conj:qveer}
We have 
$\alpha \ll_{\ri} \beta$ if and only if $\alpha \prec_{\ri} \beta$ and there exist no boundary right $P$-bigons from $\alpha$ to $\beta$.
\end{conjecture}

We study more properties of $\ll_\ri$:

\begin{lemma}\label{lemma:f}
Let $f: \mathcal{A}_{C}(S,P) \rightarrow \mathcal{A}_{C}(S)$ be the forgetful map. % induced by the obvious inclusion $S \setminus P \hookrightarrow S$. 
If $\alpha \ll_\ri \beta$ in $\A$ then we have $f(\alpha) \prec_\ri f(\beta)$ in $\mathcal A_C(S)$. 
\end{lemma}

\begin{proof}
Since $\alpha \ll_{\ri} \beta$, there is a sequence of arcs $\alpha=\gamma_0 \prec_{\ri} \gamma_1 \prec_{\ri} \cdots \prec_{\ri} \gamma_n = \beta$ in $\mathcal{A}_{C}(S,P)$ with $\Int(\gamma_{i}) \cap \Int(\gamma_{i+1})=\emptyset$ for all $i$. 
This implies that  
$\gamma_i$ and $\gamma_{i+1}$ do not cobound any marked bigons.
Therefore, $\Int (f(\gamma_{i})) \cap \Int(f(\gamma_{i+1}))=\emptyset$ and we can conclude
$f(\alpha)=f(\gamma_0) \prec_{\ri} f(\gamma_1) \prec_{\ri} \cdots \prec_{\ri} f(\gamma_n) = f(\beta)$ in $\mathcal A_C(S)$; that is, $f(\alpha) \prec_{\ri} f(\beta)$ in $\mathcal A_C(S)$.
\end{proof}

\begin{remark-no-number}
The converse of Lemma~\ref{lemma:f} does not hold in general, even if we assume $\alpha \prec_{\ri} \beta$. See Figure \ref{fig:porder} (c). 
\end{remark-no-number}

The next proposition gives a sufficient condition for $\alpha\ll_\ri\beta$. 

\begin{proposition}
\label{proposition:sufficient-condition}
Let $\alpha,\beta \in \mathcal{A}_{C}(S,P)$ be arcs with $\alpha \prec_{\ri} \beta$. If $\alpha$ and $\beta$ do not cobound bigons with marked points %(namely, if $\alpha$ and $\beta$, viewed as arcs in the non-punctured surface $S$, intersect efficiently), 
then $\alpha \ll_\ri \beta$.
\end{proposition}

\begin{proof}

If $\alpha$ and $\beta$ do not cobound bigons with marked points  then following the proof of \cite[Lemma 5.2]{hkm} one can construct an arc $\gamma \in \A$ such that $\alpha \prec_{\ri} \gamma \prec_{\ri} \beta$ with $\#(\alpha,\gamma) < \#(\alpha,\beta)$ and $ \#(\gamma,\beta)<\#(\alpha,\beta)$. Here $\#(-,-)$ denotes the geometric intersection number of the interiors of the two arcs.
Moreover, the construction of $\gamma$ shows that $\alpha$ and $\gamma$ ($\gamma$ and $\beta$) do not cobound bigons with marked points. Thus iterating this interpolation process, we get a sequence of arcs satisfying the conditions (\ref{eqn:defn-ll}) and (\ref{eqn:defn-ll2}).
\end{proof}

\subsection{Definition of quasi-right-veering}

The mapping class group $\MCG(S, P)$ acts on the set $\A$ naturally. Let $\phi\in\MCG(S,P)$ be represented by $\vphi\in\Diff(S,P,\partial S)$ and $\alpha\in\A$ be represented by an arc $a\in S\setminus P$. Then $\phi(\alpha)$ denotes the isotopy class of the arc $\vphi(a)$.

Naturally extending the notion of right-veering mapping classes in \cite{hkm}, we define the following, cf. \cite[p.949]{bvv}:

\begin{definition}[Right-veering] 
\label{defn:right-veering}
We say that 
$\psi \in \MCG(S,P)$ is \emph{right-veering} with respect to the boundary component $C$ if $\alpha \prec_{\sf right} \psi(\alpha)$ or $\alpha = \psi(\alpha)$ for every $\alpha \in \mathcal{A}_{C}(S,P)$.
\end{definition}

\begin{remark}
In \cite{bg,pl}, a slightly different definition of ``right-veering'' is used. In Section~\ref{sec:comparison} we discuss the relationship between these two superficially different notions of right-veering.
\end{remark}

Since $\prec_{\sf right}$ is a total ordering on the set $\mathcal{A}_{C}(S,P)$, $\psi\in MCG(S,P)$ is right-veering if and only if $\psi(\alpha) \not\prec_{\sf right} \alpha$ for every $\alpha \in \mathcal{A}_{C}(S,P)$. 

Hinted at this alternative definition of right-veering, we introduce quasi-right-veering mapping classes. 

\begin{definition}[Quasi-right-veering]\label{defn:qveer}
We say that $\psi \in \MCG(S,P)$ is \emph{quasi-right-veering} with respect to the boundary component $C$ of $S$ if every  arc $\alpha \in \mathcal{A}_{C}(S,P)$ satisfies $\psi(\alpha) \not\ll_{\sf right} \alpha$.
(Warning: Since ``$\ll_{\sf right}$'' is not a total ordering, $\psi(\alpha) \not\ll_{\sf right} \alpha$ is not equivalent to $\alpha  \ll_{\sf right} \psi(\alpha)$ or $\alpha  = \psi(\alpha)$.)

%\item
%Let $L$ be a closed braid with respect to an open book $(S,\phi)$ {\color{red} satisfying $p(L\cap S_0)\subset \nu(\partial S)$. }
%We say that $L$ is \emph{quasi-right-veering} with respect to a boundary component $C$ if its distinguished monodromy $\phi_{L} \in \MCG(S,P)$ is quasi-right-veering with respect to $C$.

%\item
%We say that $L$ is \emph{quasi-right-veering} if $L$ is quasi-right-veering with respect to every boundary component of $S$. 
%\end{itemize}
\end{definition}

We note that the definitions of ``right-veering'' and ``quasi-right-veering'' are independent of a choice of the distinguished point $*_C$. 

We show that for a distinguished monodromy of closed braids, being (quasi)-right-veering is well-defined.

\begin{proposition}
\label{prop:well-defined}
Let $L$ (resp. $L'$) be a closed braid with respect to an abstract open book $(S,\varphi)$ (resp. $(S, \vphi')$) satisfying $P=p(L \cap S_0) \subset \nu(\partial S)$  (resp. $P'=p(L' \cap S_{0}) \subset \nu(\partial S)$).  

Suppose that $((S, \vphi), L)$ and $((S, \vphi'), L')$ are equivalent.  
Then for every boundary component $C$ of $S$, the distinguished monodromy $[\vphi_L]$ is right-veering (resp. quasi-right-veering) with respect to $C$ if and only if the distinguished monodromy $[\vphi'_{L'}]$ is right-veering (resp. quasi-right-veering) with respect to $C$.
\end{proposition}

\begin{proof}
A diffeomorphism $\theta \in \Diff(S,\partial S)$ induces a map $\theta_{*}: \mathcal{A}_{C}(S,P) \rightarrow \mathcal{A}_C(S,\theta(P))$. 
By definition of $\prec_{\sf right}$ and $\ll_{\sf right}$, both $\prec_{\sf right}$ and $\ll_{\sf right}$ are preserved by $\theta_{*}$. That is, $\alpha \prec_{\sf right} \beta$ (resp. $\alpha \ll_{\sf right} \beta$) if and only if $\theta_{*}(\alpha)  \prec_{\sf right} \theta_{*}(\beta)$ (resp. $\theta_*(\alpha) \ll_{\sf right} \theta_*(\beta)$).

This implies that, if $\Theta: \MCG(S,P) \rightarrow \MCG(S,P')$ is a point-changing isomorphism (Definition~\ref{def:point-changing}) then  $\phi \in \MCG(S,P)$ is right-veering (resp. quasi-right-veering) if and only if $\Theta(\phi) \in \MCG(S,P')$ is right-veering (resp. quasi-right-veering). 
By Theorem~\ref{thm:well-defined-up-to-point},  %Proposition \ref{prop:isotopic-braids-1}, 
this means that the distinguished monodromy $[\vphi_L]$ is right-veering (resp. quasi-right-veering) if and only if $[\vphi'_{L'}]$ is right-veering (resp. quasi-right-veering).
\end{proof}

%\marginpar{\tiny one definition was removed} 

%\begin{definition}
%Let $(S,\phi)$ be an open book decomposition of $M$, which we mean that there is an abstract open book $(S,\varphi)$ with $[\varphi]=\phi$ and that $M_{(S,\varphi)}$ is diffeomorphic to $M$.

%We say that an oriented link in $L$ in $M$ is a closed braid with respect to an open book decomposition $(S,\phi)$, if one may choose a diffeomorphism $\tau:M \rightarrow M_{(S,\varphi)}$ so that $\tau()$

%As is well-known, $M_{(S,\varphi)}$ admits a contact structure $\xi_{(S,\varphi)}$ which is unique up to isotopy. We say that $(S,\phi)$ is an open book decomposition of $(M,\xi)$ if $(M_{(S,\varphi)},\xi_{(S,\varphi)})$ is contactomorphic to $(M,\xi)$ 
%\end{definition}

%By the definition, it is easy to see that if closed braids $L$ and $L'$ with respect to the same abstract open book $(S, \vphi)$ are braid isotopic then $L$ and $L'$ are equivalent. 

Now we define a (quasi)-right-veering closed braid which is a central object in the paper.

\begin{definition}[Right-veering/quasi-right-veering closed braid]
Let $C$ be a boundary component of $S$. 
Let $L$ be a closed braid with respect an abstract open book $(S, \vphi)$. 
We say that
the closed braid $\cL:= [(S, \vphi), L)]$ with respect to the open book $(S,\phi)$ is 
\begin{itemize}
\item[--]
{\em right-veering with respect to $C$} if there exists a closed braid $L'$ with respect to $(S, \vphi)$ that represents $\cL$ such that the conditions (\ref{eqn:collar}) and (\ref{eqn:P}) are satisfied and  $[\vphi_{L'}] \in \MCG(S,P)$ is right-veering with respect to $C$. 

\item[--]
\emph{quasi-right-veering with respect to $C$} if there exists a closed braid $L'$ with respect to $(S, \vphi)$ that represents $\cL$ such that the conditions (\ref{eqn:collar}) and (\ref{eqn:P}) are satisfied and  $[\phi_{L'}] \in \MCG(S,P)$ is quasi-right-veering with respect to $C$.   
 
\item[--] 
\emph{right-veering} if $\cL$ is right-veering with respect to every boundary component of $S$. 

\item[--] 
\emph{quasi-right-veering} if $\cL$ is quasi-right-veering with respect to every boundary component of $S$. 
\end{itemize}   
Well-defindness follows by Proposition \ref{prop:well-defined}. 
\end{definition}

\subsection{Comparison of quasi-right-veering and right-veering}\label{subsection:QRV-RV}

In this section, we discuss relation (Proposition~\ref{lemma:rvtoqrv}) and difference (Proposition~\ref{prop:FDTCvsqrv} and Corollary~\ref{notmonoid}) between quasi-right-veering and right-veering.  

First, if $L$ is empty then by identifying $[\vphi_{L}]$ with $\phi$, the empty closed braid is quasi-right-veering if and only if the monodromy $\phi$ is right-veering. 

In general, we have the following.

\begin{proposition}
\label{lemma:rvtoqrv}
Let $\psi \in \MCG(S,P)$ be a mapping class. 

%A mapping class $\psi \in \MCG(S,P)$ is quasi-right-veering if $\psi$ is right-veering. %More generally, $\psi \in \MCG(S,P)$ is quasi-right-veering if $f(\psi) \in \MCG(S)$ is right-veering, where 
\begin{enumerate}
\item
If $\psi$ is right-veering then $\psi$ is quasi-right-veering. 
\item
If $f(\psi) \in \MCG(S)$ is right-veering then $\psi$ is quasi-right-veering, where 
$f:\MCG(S,P) \rightarrow \MCG(S)$ is the forgetful map in the generalized Birman exact sequence $(\ref{eqn:Birman})$. 
\end{enumerate}
\end{proposition}    

\begin{corollary}\label{cor:about-braid}
Every closed braid with respect to an open book $(S,\phi)$ is quasi-right-veering if $\phi\in \MCG(S)$ is right-veering. 
In particular, every closed braid with respect to the open book $(D^2, id)$ is quasi-right-veering.
\end{corollary}

\begin{proof}[Proof of Proposition~\ref{lemma:rvtoqrv}]
The first statement immediately follows from the definition of quasi-right-veering. 

To prove the second statement, assume that $\psi\in \MCG(S,P)$ is not quasi-right-veering with respect to some boundary component $C$ of $S$. 
Then there exists an arc $\alpha\in\mathcal{A}_{C}(S,P)$ such that $\psi(\alpha)\ll_\ri\alpha$.  
By Lemma~\ref{lemma:f} we get 
$f(\psi) (f(\alpha)) = f(\psi(\alpha)) \prec_\ri f(\alpha)$ in $\mathcal A_C(S)$; 
that is, $f(\psi) \in \MCG(S)$ is not right-veering. 
\end{proof}

It is proved in \cite[Section 3]{hkm} that the right-veeringness of $\phi \in \MCG(S)$ is almost equivalent to positivity of its FDTC. 
We say ``almost'' because the statement is slightly complicated  when the FDTC $=0$ for non-pseudo Anosov case . 
If $\phi \in \MCG(S)$ is pseudo Anosov, $\phi$ is right-veering with respect to a boundary component $C$ if and only if $c(\phi,C)>0$. 
We remark that parallel statements on positivity and right-veering-ness hold for elements $\psi \in \MCG(S,P)$. 
Namely if $\psi$ is right-veering then $c(\psi, C) \geq 0$. 
Moreover, if $\psi$ is pseudo Anosov then $\psi$ is right-veering with respect to $C$ if and only if $c(\psi,C)>0$.

The next proposition shows significant difference between quasi-right-veering and right-veering. In particular, quasi-right-veering is much less related to positivity of the FDTC.

\begin{proposition}
\label{prop:FDTCvsqrv}
Let $(S,\phi)$ be an open book. 
\begin{enumerate}
\item 
For every boundary component $C$ of $S$ and integers $N <0$ and $n >1$, there exists a closed $n$-braid $\cL=[(S, \vphi), L)]$ with respect to $(S,\phi)$ such that 
\begin{itemize}
\item
$\cL$ is quasi-right-veering with respect to $C$, and 
\item
$c(\phi, \cL, C) \leq N<0$; i.e., $\cL$ is non-right-veering with respect to $C$.
\end{itemize}
\item 
For every negative integer $N$ there exists a closed braid $L$ with respect to $(S,\phi)$ such that  
\begin{itemize}
\item
$\cL$ is quasi-right-veering, and 
\item
$c(\phi,\cL, C) \leq N<0$ for every boundary component $C$; i.e., $\cL$ is non-right-veering.  
\end{itemize}
\end{enumerate}
\end{proposition}

\begin{proof}
Fix a boundary component $C$ of $S$. 
Take  $\varphi \in \Diff(S, \partial S)$ representing $\phi$ so that $\varphi|_{\nu(\partial S)}=id$. 
Let $\nu(C)$ denote the connected component of the $\nu(\partial S)$ that contains $C$. 
We identify $\nu(C)$ with the annulus $A= \{z \in \C \: | \: 1 \leq |z| < 2\}$ so that the boundary component $C$ is identified with $\{z \in \C \: | \: |z|=1\}$. 

We put $$P = \left\{p_i \in \C \: \middle| \: i=1,\dots,n \mbox{ and } p_{i}=  1+\frac{i}{n+1} \right\} \subset A \cong \nu(C) \subset S.$$
For $k\in\mathbb N$ let $\beta_{C,k} \subset S \times [0,1] $ be the geometric $n$-braid whose $i$-th strand $\gamma_{k,i}:[0,1] \rightarrow A\times [0,1] \subset S \times [0,1]$ is given by  (see Figure~\ref{fig:qrvexam}-(1))  
\[
\gamma_{k,i}(t) = 
\begin{cases}
((1+\frac{1}{n+1})\exp(2\pi \sqrt{-1}k t),\ t)  & (i=1) \\
((1+\frac{2}{n+1})\exp(-2\pi \sqrt{-1}k t), \ t) & (i=2) \\
(1+\frac{i}{n+1}, \ t ) & (i = 3,\dots,n). 
\end{cases}
\] 
Thus, the 1st strand of $\beta_{C,k}$ winds $k$ times around $C$  counterclockwise and the 2nd strand winds $k$ times  clockwise.  
Let $L_{C,k}:= \pi(\beta_{C,k}) \subset M_{(S,\varphi)}$ be the closed $n$-braid with respect to the abstract open book $(S,\vphi)$  obtained by taking the braid closure of $\beta_{C,k}$, where $\pi:S\times[0,1] \rightarrow M_{(S,\varphi)}:= S\times[0,1]/\sim$ is the quotient map. 

\begin{figure}[htbp]
\begin{center}
\includegraphics*[bb= 148 546 449 715,width=90mm]{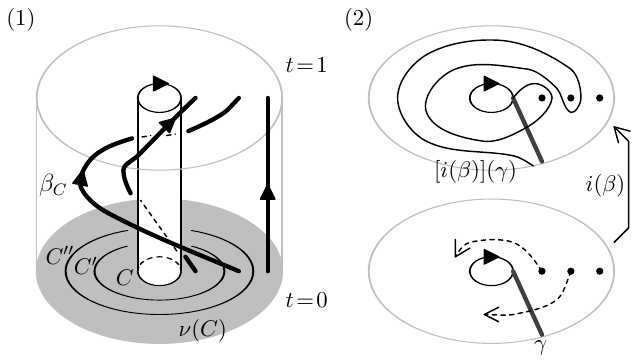}
\caption{
(1) The braid $L_{C, 1}$ is not right-veering but quasi-right-veering. \\
(2) The map $(T_{C})^{-1}(T_{C'})^{2}(T_{C''})^{-1}$ forces to form a boundary right $P$-bigon.}
\label{fig:qrvexam}
\end{center}
\end{figure}

With the push map $i: B_n(S) \rightarrow \MCG(S,P)$ in the generalized Birman exact sequence (\ref{eqn:Birman}) we have $i([\beta_{C,1}])= 
(T_{C})^{-1} (T_{C'})^{2} (T_{C''})^{-1}$, where $T_{C}, T_{C'}$ and $T_{C''}$ are the right-handed Dehn twists along the curves $C$, $C' = \{z \in A \: | \: |z|= \frac{3}{2n+2}\}$ and $C''= \{z \in A \: | \: |z|= \frac{5}{2n+2}\}$.
The distinguished monodromy of the closed braid $L:=L_{C,k}$ is 
$$
[\vphi_L]=[\vphi_{L_{C,k}}]= i([\beta_{C,k}])[j(\vphi)] = (T_{C})^{-k} \ (T_{C'})^{2k} \ (T_{C''})^{-k} [j(\varphi)]
$$ 
Since $j(\varphi)=id$ on $\nu(C)$ we have $c(\phi, \cL, C) = c([\vphi_L], C)= -k <0$. 
This shows that $\cL$ is not right-veering.

For any $\gamma \in \A$ the factor $(T_{C})^{-k}(T_{C'})^{2k}(T_{C''})^{-k}$ of $[\vphi_L]$ forces to form a boundary right $P$-bigon from $[\vphi_L](\gamma)$ to $\gamma$. See Figure~\ref{fig:qrvexam}-(2). 
Thus by Proposition~\ref{lemma:b-right-p-bigon} 
$[\vphi_L](\gamma) \not\ll_\ri \gamma$ for every $\gamma \in \A$, which means $\cL$ is quasi-right-veering with respect to $C$. This proves (1).

Next we prove (2).
Let $\{C_1,\dots, C_d\}$ be the set of boundary components of $S$. For each component $C_{i}$ we take a closed braid $L_{C_{i},k}$ given in the proof of (1), and let $L = \bigsqcup_{i=1}^{d} L_{C_{i},k}$ be the disjoint union of $L_{C_{i},k}$. %By (1) and Proposition~\ref{prop:isotopic-braids} $L$ is quasi-right-veering and $c(L,\phi,C_i) \leq -k$ for all $i=1,\dots,d$.  
By (1) we see that $\cL$ is quasi-right-veering. 
By Proposition~\ref{prop:isotopic-braids} we obtain $c(\phi, \cL, C_i) \leq -k$ for all $i=1,\dots,d$.  
\end{proof}

The set of right-veering mapping classes in $\MCG(S,P)$ forms a monoid. 
However, this is not the case for quasi-right-veering mapping classes: 

\begin{corollary}\label{notmonoid}
The set of quasi-right-veering mapping classes in $\MCG(S,P)$ does \emph{not} form a monoid. 
\end{corollary}

\begin{proof}
We use the same notations in Proposition~\ref{prop:FDTCvsqrv}. 
Let $\chi ={(T_{C'})}^{-1} i([\beta_{C,1}])^{-1} = T_{C}T_{C'}^{-3}T_{C''}$ and $\psi = i([\beta_{C,1}])$.
Both $\chi$ and $\psi$ are quasi-right-veering but $\chi\psi = (T_{C'})^{-1}$ is not quasi-right-veering. 
\end{proof}

\subsection{Transverse links are right-veering and quasi-right-veering}\label{subsection:transv-links}

We have defined right-veering and quasi-right-veering and studied their properties. 
In this section, we study transverse links in contact manifolds from the view point of right-veering and quasi-right-veering and obtain Propositions~\ref{prop:transverse-link} and \ref{prop:stabilization}.

Recall that an abstract open book $(S, \vphi)$ gives a natural open book decomposition of the manifold $M_{(S, \vphi)}$ (see Section~\ref{subsec:2.2}). 
We say that 
a contact structure $\xi$ on $M_{(S, \vphi)}$ is {\em supported by $(S, \vphi)$} if $\xi$ is isotoped through contact structures so that there is a contact 1-form $\alpha$ for $\xi$ such that $d\alpha$ is a positive area form on each page $S_t$ of the open book and $\alpha>0$ on the binding $B$.
By Thurston and Winkelnkemper \cite{TW}, for every $(S, \vphi)$  there exists a contact structure on $M_{(S, \vphi)}$ supported by $(S, \vphi)$. Such a contact structure is unique up to isotopy due to Giroux \cite{gi} and denoted by $\xi_{(S,\vphi)}$.

\begin{definition}
Let $\xi_{(S, \vphi)}$ be a contact structure on $M_{(S, \vphi)}$ supported by $(S, \vphi)$. 
In this paper, we say that:
\begin{itemize}
\item 
a contact 3-manifold $(M, \xi)$ is {\em supported by} $(S,\vphi)$ if $(M, \xi)$ and $(M_{(S, \vphi)}, \xi_{(S, \vphi)})$ are contactomorphic. 
\item
an open book $(S,\phi)$ is an \emph{open book decomposition} of $(M,\xi)$ if $(M,\xi)$ is supported by an abstract open book $(S,\vphi)$ with $[\vphi]=\phi$.
\end{itemize}
\end{definition}

Next we list basic facts about transverse links and closed braids.  The fact (3) is discovered by Bennequin \cite{Ben} (for $(S,\phi)=(D^2, id)$ case), Mitsumatsu and Mori \cite{MM}, and Pavelescu \cite[Theorem 3.2]{pav}: 
\begin{enumerate}
\item 
Every closed braid with respect to $(S,\vphi)$ is a transverse link in $(M_{(S,\vphi)},\xi_{(S,\vphi)})$ for some contact structure $\xi_{(S,\vphi)}$ supported by $(S,\vphi)$.\\

\item The transverse link type in (1) only depends on the equivalence class of the closed braid, in the following sense:\\
 
Let $L$ and $L'$ be closed braids with respect to $(S,\vphi)$ and $(S,\vphi')$, respectively and assume that $((S,\vphi),L)$ and $((S,\vphi'),L')$ are equivalent.
Suppose that $\xi_{(S,\vphi)}$ (resp. $\xi'_{(S,\vphi)}$) is a contact structure on $M_{(S,\vphi)}$ (resp. $M_{(S,\vphi')}$) supported by $(S,\vphi)$ (resp. $(S,\vphi')$) so that $L$ (resp. $L'$) is a transverse link in $(M_{(S,\vphi)},\xi_{(S,\vphi)})$ (resp. $(M_{(S,\vphi')},\xi_{(S,\vphi')})$).

Since $\vphi$ and $\vphi'$ are isotopic, an isotopy between $\vphi$ and $\vphi'$ induces a diffeomorphism $\rho: M_{(S,\vphi)} \rightarrow M_{(S,\vphi')}$ (see (\ref{eq:diffeo-rho})) that preserves the pages. In particular, $\rho_*(\xi_{(S,\vphi)})$ is supported by the open book $(S,\vphi')$ hence it is isotopic to $\xi_{(S,\vphi')}$. By Gray stability we have a diffeomorphism $\theta: M_{(S,\vphi')} \rightarrow M_{(S,\vphi')}$ isotopic to the identity satisfying $\theta_*(\rho_*\xi_{(S,\vphi)})=\xi_{(S,\vphi')}$. Consequently, we have a contactomorphism
\begin{equation}\label{eq:vrho}
\varrho=\theta\circ\rho:  (M_{(S,\vphi)},\xi_{(S,\vphi)}) \rightarrow (M_{(S,\vphi')},\xi_{(S,\vphi')}) 
\end{equation}
and $\varrho(L)$ and $L'$ are transversely isotopic. 
Note that $\varrho(L)$ is a transverse link in $(M_{(S,\vphi')},\xi_{(S,\vphi')})$ but may not be in braid position with respect to $(S,\vphi')$ since $\theta$ may not preserve the pages.\\

\item
Any transverse link in a contact 3-manifold 
$(M_{(S,\vphi)},\xi_{(S,\vphi)})$ can be transversely isotoped to a closed braid with respect to $(S, \vphi)$.\\
\end{enumerate}

\begin{definition} 
Suppose that $(S,\phi)$ is an open book decomposition of $(M,\xi)$. 
We say that {\em a transverse link $\mathcal T$ in $(M,\xi)$ is represented by a closed braid $\cL$ with respect to $(S,\phi)$}, if there is an abstract open book $(S,\vphi)$ with $[\vphi]=\phi$ and a closed braid $L$ with respect to $(S,\vphi)$  so that $\cL=[(S,\vphi),L]$ and there is a contactomorphism  $\tau: (M,\xi) \to (M_{(S, \vphi)},\xi_{(S,\vphi)})$ such that $L=\tau(\mathcal T)$.   
\end{definition}

We prove two propositions. 
The first one is on quasi-right-veering-ness of transverse links.

\begin{proposition}\label{prop:transverse-link}
Every transverse link in a contact manifold $(M,\xi)$ admits a quasi-right-veering closed braid representative with respect to some open book decomposition of $(M,\xi)$. 
\end{proposition}

\begin{proof}
In \cite[Proposition 6.1]{hkm} of Honda, Kazez, and Mati\'c show that every contact 3-manifold admits an open book decomposition $(S,\phi)$ with right-veering monodromy. This fact and our Corollary~\ref{cor:about-braid} yield the proposition. 
\end{proof}

The second proposition is about right-veering-ness of transverse links.

To state the proposition, we recall a positive stabilization of a closed braid. Here we present an algebraic formulation so that the connection to distinguished monodromy is clear. For a geometric formulation based on open book foliation machinery, we refer the paper \cite{ik7}.

As usual, take $\varphi \in \Diff(S,\partial S)$ with $[\varphi]=\phi$ so that $\varphi|_{\nu(\partial S)}=id$. 
Let $L$ be a closed $n$-braid with respect to an abstract open book $(S,\varphi)$ such that $P =p(L\cap S_0) \subset \nu(\partial S)$. 
Let $C$ be a boundary component of $S$. 
Let $\nu'(C) \subset \nu(C)$ be a sub-collar neighborhood of $C$ that does not intersect $P$. 
Let $\beta_L \subset S \times[0,1]$ be the geometric $n$-braid obtained from $L$. By taking $\nu'(C)$ sufficiently small, we may assume that
\begin{equation}
\label{eqn:stab}
\beta_L \cap (\nu'(C)\times[0,1])=\emptyset.
\end{equation}
Choose a point $q \in \nu'(C)$. 
The disjoint union of the strand $\{q\}\times[0,1]$ and $\beta_L$ yields a geometric $(n+1)$-braid  
$$
\overline{\beta_L} := \beta_L \sqcup (\{q\} \times[0,1]) \subset S \times[0,1].
$$

Let $\gamma$ be a properly embedded arc in $S\setminus(P\cup \{q\})$ that connects a point $p \in P$ and $q$. 
Let $H_{\gamma} \in \MCG(S,P\cup\{q\})$ be the positive half twist about $\gamma$ and $h_{\gamma} \subset S\times[0,1]$ be a geometric $(n+1)$-braid that represents $H_{\gamma}$ 
in the sense that $i([h_{\gamma}])=H_{\gamma}$, where $i:B_{n+1}(S) \rightarrow \MCG(S,P\cup\{q\})$ denotes the push map in the generalized Birman exact sequence (\ref{eqn:Birman}).

%which we mean that $i([\beta_{\gamma}])=H_{\gamma}$, where $i:B_2(S) \rightarrow \MCG(S,\{p,q\})$ denotes the push map.
%We extend $\beta_{\gamma}$ as the $(n+1)$-braid $\overline{\beta_{\gamma}}$ by adding $(n-1)$ strands $\sqcup_{p' \in P\setminus \{p\}} p'\times [0,1]$. 

\begin{definition}[Braid stabilization]\label{def-of-L'}
Let $L'$ be a closed $(n+1)$-braid obtained by taking the braid closure of the geometric $(n+1)$-braid 
$$h_\gamma \bullet \overline{\beta_L}  \mbox{ \ (read from right to left) }$$
where the bullet ``$\bullet$'' denotes the concatenation of geometric braids. 
We say that $L'$ is a \emph{positive stabilization} of the closed braid $L$ about the arc $\gamma$.  
\end{definition}

Since $q\in\nu'(C)$ the property (\ref{eqn:stab}) implies that $L'$ is obtained by connecting $L$ and a meridian circle of $C$ with a positively twisted band. Therefore, $L$ and $L'$ are transversely isotopic. 
See \cite[Theorem 4.2]{pav} where Pavelescu proves that closed braids are transversely isotopic if and only if they differ by braid isotopies and positive stabilizations and their inverses.

Recall the diffeomorphism $\varphi_L \in \Diff(S,P,\partial S)$ in (\ref{def of varphi_L}) that represents the distinguished monodromy $[\vphi_L]\in\MCG(S, P)$. 
By (\ref{eqn:stab}), we may assume that $\varphi_L|_{\nu'(C)}=id$.
Since $q \in \nu'(C)$ we may view $\varphi_L$ as an element of $\Diff(S, P\cup \{q\},\partial S)$ and denoted it by $\overline{\varphi_L}$. 
We obtain $[\overline{\varphi_L}] \in \MCG(S,P\cup\{q\})$ and the distinguished monodromy $L'$ satisfies 
\begin{equation}
\label{eqn:stab-monodromy}
[\vphi_{L'}] = H_{\gamma}[\overline{\vphi_L}] \in \MCG(S,P\cup\{q\}).
\end{equation}

Here is the second proposition:

\begin{proposition}
\label{prop:stabilization}
Every closed braid $\cL=[L]$ with respect to an open book $(S,\phi)$ can be made right-veering after a sequence of positive  stabilizations of $L$.
\end{proposition}

When $(S,\phi) = (D^{2},id)$ the same statement is proved in \cite[Proposition 3.1]{pl}.

\begin{proof}

Suppose that $L$ is a closed braid with respect to an abstract open book $(S, \vphi)$ so that $L$ represents $\cL$ and the conditions (\ref{eqn:collar}) and (\ref{eqn:P}) are satisfied.

Let $C$ be a boundary component of $S$. 
Let $\nu'(C) \subset \nu(C)$ be a sub-collar neighborhood of $C$ that does not contain (or intersect) $P=p(L\cap S_0) \subset \nu(\partial S)$.
Choose points $q$ and $q'$ in $\nu'(C)$. 
Let $\gamma_1 \subset S \setminus (P\cup\{q, q'\})$ be a properly embedded arc that connects one of the points in $P$ and the point $q$.
Let $\gamma_2 \subset \nu'(C)\setminus \{q, q'\}$ be an arc that connects $q$ and $q'$. Assume the following (see Figure~\ref{fig:stab}). 
\begin{enumerate}
\item 
The interiors of $\gamma_1$ and $\gamma_2$ intersect exactly at one point in $\nu'(C)$. We name it $r$. 
\item 
Let $\gamma_1' \subset \gamma_1$ and $\gamma_2' \subset \gamma_2$ be the sub-arcs connecting $r$ and $q$. 
Then the simple closed curve $\gamma_1' \cup \gamma_2'$ is homotopic to $C$ in $S\setminus (P\cup\{q'\})$. 
\end{enumerate}
\begin{figure}[htbp]
\begin{center}
\includegraphics*[bb=206 629 398 713,width=70mm]{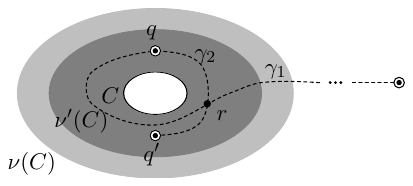}
\caption{Twice stabilizations about $C$ makes a closed braid right-veering with respect to $C$.}
\label{fig:stab}
\end{center}
\end{figure}

Let $L'$ be a closed $(n+2)$-braid obtained from $L$ by positive stabilizations first about $\gamma_1$ and then $\gamma_2$ as constructed in Definition~\ref{def-of-L'}. 

The diffeomorphism $\vphi_L \in \Diff(S,P,\partial S)$ satisfies $\varphi_L= id$ on $\nu'(C)$. 
Let $\overline{\vphi_L}$ denote the diffeomorphism $\vphi_L$ viewed as an element of $\Diff(S,P\cup\{q,q'\}, \partial S)$. 
By (\ref{eqn:stab-monodromy}) the distinguished monodromy of $L'$ satisfies 
$[\vphi_{L'}]= H_{\gamma_2}  H_{\gamma_1} [\overline{\vphi_L}] \in \MCG(S,P\cup\{q,q'\})$.
%Here $\overline{\phi_L}$ is the distinguished monodromy of $L$ viewed as an element of $\MCG(S,P\cup\{q,q'\})$ by taking a representative $\varphi_L \in \Diff(S,P,\partial S)$ with $\varphi_L|_{\nu'(C)}=\textrm{id}$.

Since $\varphi_L={\rm id}$ on $\nu'(C)$ 
and every essential arc in $\mathcal A_C(S, P\cup\{q,q'\})$ intersects either $\gamma_{1}$ or $\gamma_2$, the monodromy $[\vphi_{L'}]$ is right-veering with respect to $C$.

Applying this operation for every boundary component we get a right-veering closed braid that is transversely isotopic to the original braid $L$. 
\end{proof}

\section{Characterization of non-loose links}
\label{sec:proof}

We now prove our main theorem: 

\begin{theorem}
\label{theorem:main}
A transverse link $\mathcal{T}$ in a contact 3-manifold $(M,\xi)$ is non-loose if and only if every braid representative of $\mathcal{T}$ with respect to every open book decomposition  of $(M,\xi)$ is quasi-right-veering. 
\end{theorem}

Our proof of Theorem~\ref{theorem:main} is a generalization of the proof of \cite[Theorem~2.4]{ik1-2}. 
We may assume that the readers are familiar with basic definitions and properties of open book foliations that can be found in \cite{ik1-1,ik2,ik3}.

\begin{proof}[Proof of Theorem \ref{theorem:main}]
($\Rightarrow$) 
First we show that non-quasi-right-veering braid is loose. 
Assume that a transverse link %$K$ 
$\mathcal{T}$ can be represented by a non-quasi-right-veering closed braid $L$ with respect to an abstract open book $(S,\vphi)$ that supports $(M, \xi)$.
Let $P:=p(L\cap S_0)$. 
That is, there exist a boundary component $C \subset \partial S$ and an arc $\alpha \in \mathcal{A}_{C}(S,P)$ such that there is a sequence of arcs $[\vphi_{L}](\alpha) = \alpha_0 \prec_{\sf right} \alpha_1 \prec_{\sf right} \cdots \prec_{\sf right} \alpha_{k} = \alpha$ with $\Int(\alpha_i) \cap \Int(\alpha_{i+1})= \emptyset$ for all $i=0,\dots,k-1$.

We explicitly construct a transverse overtwisted disk $D_{\sf trans}$ in $M_{(S,\vphi)} \setminus L$ %$M\setminus L$ 
by giving its movie presentation. 
A similar construction can be found in \cite{ik1-2}. 
Here, a {\em transverse overtwisted disk} (see \cite[Definition 4.1]{ik1-1} for the precise definition) is a disk admitting a certain type of open book foliation and is bounded by a transverse push-off of a usual overtwisted disk. 

For $i=0,\ldots,k$ denote the endpoint $\alpha_{i}(1) \in \partial S$ of the arc $\alpha_i$ by $w_i$.
Slightly moving $w_i$ along $\partial S$, if necessary, we may assume that all the points $w_0,\dots,w_{k-1}$ are distinct and still satisfying $\Int(\alpha_i) \cap \Int(\alpha_{i+1})= \emptyset$. Since $[\varphi_L](\alpha)=\alpha_0$ we get $w_0=w_k$. 
Fix a sufficiently small $\varepsilon>0$.
\begin{center}
\begin{figure}[htbp]
\includegraphics*[width=70mm]{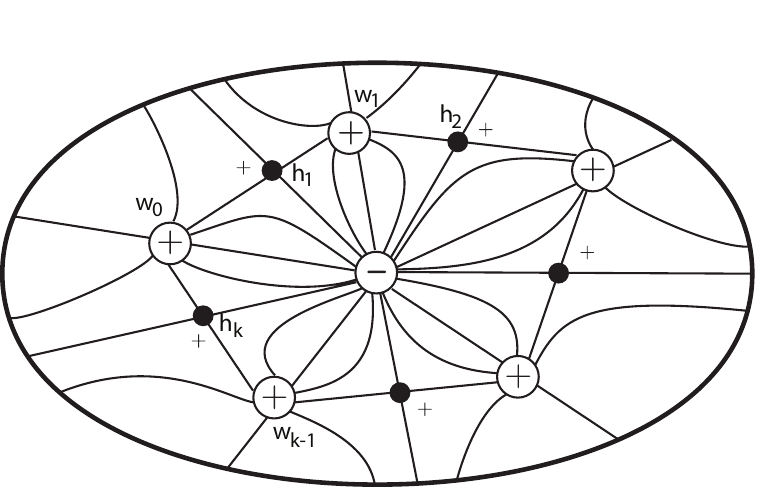}
\caption{Transverse overtwisted disk $D_{\sf trans}$.}
\label{fig:ot-disk}
\end{figure}
\end{center}

The open book foliation of $D_{\sf trans}$ contains 
one negative elliptic point at $\ast_{C}$ and 
$k$ positive elliptic points at $w_0,\dots,w_{k-1}$.  

The movie presentation of $D_{\sf trans}$ on the page $S_0$ 
consists of $(k-1)$ a-arcs emanating from $w_1,\dots,w_{k-1}$ and a b-arc that is a copy of $\alpha_0$ joining $w_0$ and $\ast_{C}$. 
For $t \in [0, \frac{1}{k+1})$ the movie presentation on the page $S_t$ is the same as $S_0$.

The movie presentation on the page $S_{\frac{1}{k+1}}$ 
contains one hyperbolic point, $h_1$, whose describing arc joining $\alpha_0$ and the a-arc from $w_1$ is a parallel copy of $\alpha_{1}$ in $S_{\frac{1}{k+1}-\e}$. 
Since $\Int(\alpha_0) \cap \Int(\alpha_1)=\emptyset$ the interior of the describing  arc is disjoint from all the a-arcs and the b-arc in the page $S_{\frac{1}{k+1}-\e}$. 
Since $\alpha_0 \prec_\ri \alpha_1$ the normal vectors of $D_{\sf trans}$ point {\em out} of the describing arc, thus by \cite[Observation 2.5]{ik4} the sign of the hyperbolic point $h_1$ is positive.
The movie presentation on the page $S_{\frac{1}{k+1}+\varepsilon}$ consists of one b-arc which is a copy of $\alpha_1$ connecting  $w_1$ and $*_C$ and $(k-1)$ a-arcs emanating from  $w_0, w_2,\dots,w_{k-1}$.

\begin{figure}[htbp]
\begin{center}
\includegraphics*[bb= 
128 602 490 710,width=120mm]{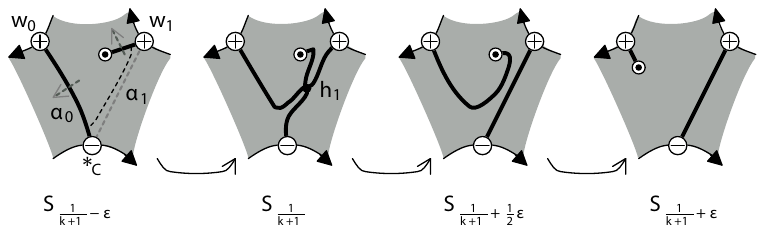}
\caption{Movie for $t\in [\frac{1}{k+1}-\varepsilon,\frac{1}{k+1}+\varepsilon]$: the b-arc $\alpha_0$ disappears and the new b-arc $\alpha_1$ appears at $t=\frac{1}{k+1}$. 
The black dashed arc is the describing arc and the gray dashed arc is $\alpha_1$. 
The black solid arrows indicate the orientations of $\partial S$. 
The gray dashed arrows are normal vectors to $D_{\sf trans}$. }
\label{fig:movie0}
\end{center}
\end{figure}

We inductively apply the above procedure.
Let $j=1,\dots, k$. 
The above paragraph describes $j=1$ case. 

On the page $S_{\frac{j}{k+1}}$ ($j>1$) we put a positive hyperbolic point $h_j$ whose describing arc is a parallel copy of $\alpha_{j}$. 
As a consequence the page $S_{\frac{j}{k+1}+\varepsilon}$ has one b-arc which is a copy of  $\alpha_j$ connecting $w_j$ and $\ast_{C}$ and $(k-1)$ a-arcs emanating from $w_i$ for $i=1,\dots,j-1, j+1,\dots,k-1$.

On the page $S_1$ the movie presentation consists of one b-arc which is a copy of $\alpha_k = \alpha$ and $(k-1)$ a-arcs emanating from $w_0,\dots,w_{k-1}$. 
Since $[\vphi_{L}](\alpha)=\alpha_0$, 
the slices $D_{\sf trans} \cap S_1$ and $D_{\sf trans} \cap S_0$ of $D_{\sf trans}$ can be identified under the distinguished monodromy $\vphi_{L}$.  
In other words the movie presentation gives rise to an embedded surface in 
$M_{(S,\vphi)} \setminus L$. 
The construction tells us that the surface is topologically a disk, and moreover it is a transverse overtwisted disk (see \cite{ik1-2}).

($\Leftarrow$) 
Assume that a transverse link $%L
\mathcal{T} \subset (M, \xi)$ is loose. By taking a neighborhood of an overtwisted disk $D \subset M\setminus %L
\mathcal{T}$, we may regard $(M,\xi)$ as the connected sum $(M',\xi') \# (S^{3},\xi'_{ot})$ such that $%L
\mathcal{T} \subset (M',\xi')$. Here $\xi'_{ot}$ denotes some overtwisted contact structure on $S^{3}$. 
Applying the argument of Honda, Kazez and Mati\'c in \cite[p.444]{hkm} to $(S^{3},\xi'_{ot})$ we may regard $(M,\xi)$ as $(N,\xi_N) \# (S^{3},\xi_{ot})$ such that  $%L
\mathcal{T} \subset (N,\xi_N)$, where $(S^{3},\xi_{ot})$ denotes the overtwisted contact structure supported by $(A,T_{A}^{-1})$ an annulus open book  with a left-handed Dehn twist about a core curve of $A$.

Take an abstract open book $(S_{N},\vphi_N)$ supporting $(N,\xi_N)$ and a closed braid $L_N$ representing %$L$. 
$\mathcal{T}$.
Then the original contact 3-manifold $(M,\xi)$ is supported by the open book $(S, \vphi) := (S_N,\vphi_N) * (A, T_A^{-1})$ where $\ast$ represents a Murasugi sum of the open books (cf. \cite[Definition 2.16]{Et}) and $L_{N}$ is in braid position with respect to $(S, \vphi)$. 

Let $\gamma$ be the isotopy class of a co-core of the attached 1-handle $S\setminus S_N$. 
We have $$[\vphi_{L_N}]\gamma= [T_A^{-1}]\gamma \ll_{\sf right} \gamma$$ hence $L_N$ is not quasi-right-veering.  
\end{proof}

\begin{corollary}
A transverse link $\mathcal{T}$ in a contact 3-manifold $(M,\xi)$ is non-loose if and only if for every closed braid representative $L$ of $\mathcal{T}$ with respect to every abstract open book $(S, \vphi)$ that supports $(M,\xi)$ and for every boundary component $C$ and %isotopy class $\gamma$ of a properly embedded arc in $S\setminus P$, 
$\gamma \in \A$ where $P:=p(L\cap S)$ at least one of the following holds:
\begin{enumerate}
\item
$\gamma=[\vphi_L](\gamma)$.
\item
$\gamma\prec_\ri[\vphi_L](\gamma)$.
\item
$\gamma$ and $[\vphi_L](\gamma)$ cobound bigons that contain points of $P$. 
\end{enumerate}
\end{corollary}

\begin{proof}
$(\Rightarrow)$ 
If there exists $\gamma$ such that $[\vphi_L](\gamma) \prec_\ri \gamma$ and no marked bigons are cobounded by $[\vphi_L](\gamma)$ and $\gamma$, then 
Proposition~\ref{proposition:sufficient-condition}
shows that $[\vphi_L](\gamma) \ll_\ri \gamma$. 
Thus $[\vphi_L] \in\MCG(S,P)$ is not quasi-right-veering.   
Then Theorem~\ref{theorem:main} shows that % $K$
$\mathcal{T}$ is loose.    

$(\Leftarrow)$
This implication holds by exactly the same proof of ($\Leftarrow$ part of) Theorem~\ref{theorem:main}. 
\end{proof}

\section{Depth of transverse links}\label{sec:depth}

Theorem \ref{theorem:main} can be used to study the \emph{depth} that measures non-looseness of  transverse links and is introduced by Baker and Onaran \cite{bo}. 

Let $F$ be an oriented surface in an oriented 3-manifold $M$ and $K\subset M$ be an oriented link that transversely intersects $F$.
We denote the number of  intersection points of $K$ and $F$ by $\#(K \cap F)$, which is not necessarily realizing the geometric intersection number. 
We also denote the number of positive and negative intersection points of $K$ and $F$ by $\#^{+}(K \cap F)$ and $\#^{-}(K \cap F)$, respectively.   
We have $\#(K \cap F) = \#^{+}(K \cap F)+\#^{-}(K \cap F)$.

\begin{definition}\cite[p.1031 and 1057]{bo}
Let $K$ be a transverse link or a Legendrian link in $(M, \xi)$. 
The depth $d(K)$ of $K$ is defined by:
\[ d(K) = \min \{\#(K \cap D)  \: | \: D \text{ is an overtwisted disk in } (M,\xi) \} \]
\end{definition}

Assuming that $(M,\xi)$ is overtwisted, we see that $K$ is loose if and only if $d(K)=0$.

In the following $K$ represents a transverse link.

First we give a new interpretation of the depth $d(K)$ in terms of open book foliations. 
Let $(S,\phi)$ be an open book supporting a contact 3-manifold $(M,\xi)$. 
Recall that existence of a \emph{transverse} overtwisted disk  in the open book $(S, \phi)$ (see \cite[Definition 4.1]{ik1-1}) is equivalent to existence of an overtwisted disk in $(M, \xi)$. 

%First we show that the depth of $K$ is equal to the minimal number of the negative intersection points of $K$ with a \emph{transverse} overtwisted disk (\cite[Definition 4.1]{ik1-1}). 
%The same result is proved in \cite{ik6} for the case when $K$ is the binding of an open book. 

\begin{theorem}
\label{theorem:depth}
For a transverse link $K$ in $(M, \xi)$ let: 
\begin{equation*}
d^-_{\sf trans}(K) :=\min\left\{ \#^{-}(K' \cap D)  \: \left| \: 
\begin{array}{l}
K' \text{ is a link transversely isotopic to } K, \\
D \text{ is a transverse overtwisted disk in } (S, \phi). 
\end{array}
\right.
\right\} 
\end{equation*}
Then $d(K)=d^-_{\sf trans}(K)$. 
\end{theorem}

A special case where $K$ is the binding of an open book the equality is proved in \cite{ik6}.

The theorem highlights the difference between a  transverse overtwisted disk (whose boundary is a transverse unknot) and a usual overtwisted disk (whose boundary is a Legendrian unknot).

Applications of the theorem can be found in Theorems~\ref{theorem:depthone} and \ref{theorem:nonloose}.  

\begin{proof}
We first show that $d(K) \leq d_{\sf trans}(K)$.

Let $D_{\sf trans}$ and $K_0$ be transverse overtwisted disk and transverse link which attain $d_{\sf trans}(K)$. Therefore, $d_{\sf trans}(K) =\#^{-}(K_0 \cap D_{\sf trans}).$
By the structural stability theorem \cite[Theorem 2.21]{ik1-1}, we may assume that 
\begin{enumerate}
\item[(a)] The characteristic foliation $\cF(D_{\sf trans})$ and the open book foliation $\F(D_{\sf trans})$ are topologically conjugate.
\end{enumerate}

Let $G_{++}(\cF(D_{\sf trans}))$ (resp. $G_{--}(\cF(D_{\sf trans}))$) be the Giroux graph \cite[Page 646]{gi} consisting of the positive (resp. negative) elliptic points and the stable (resp. unstable) separatrices of positive (resp. negative) hyperbolic points. 
By the assumption (a), these graphs are identified with the corresponding graphs $G_{++}:= G_{++}(\F(D_{\sf trans}))$ and $G_{--}:=G_{--}(\F(D_{\sf trans}))$  in the open book foliation $\F(D_{\sf trans})$, see \cite[Definition 2.17]{ik1-1} for the definitions.

Take small neighborhoods $N_{+}, N_{-} \subset D_{\sf trans}$ of the graphs $G_{++}(\cF(D_{\sf trans}))$ and $G_{--}(\cF(D_{\sf trans}))$, repsectively. 
By transverse isotopy we move $K_0$ without introducing new intersection points with $D_{\sf trans}$ so that:  
\begin{enumerate}
\item[(b)]
The intersection $K_0 \cap D_{\sf trans}$ is disjoint from the region $N_{+} \cup N_{-}$.
\end{enumerate} 

We apply Giroux elimination lemma \cite[Lemma 3.3]{Gconvex} to remove all the positive elliptic and positive hyperbolic points of $\cF(D_{\sf trans})$ (see Figure~ \ref{fig:transtousual}). 
Call the resulting disk $D'$. 
By (a) and the definition of a transverse overtwisted disk, the characteristic foliation $\cF(D')$ has a unique negative elliptic point enclosed by a circle leaf.  
We can find a usual overtwisted disc $D \subset D'$ . 
Since the Giroux elimination is supported on $N_{+} \cup N_{-}$, the condition (b) implies that this process does not produce new intersections, i.e., $K_0 \cap D_{\sf trans} = K_0 \cap D'$. 

\begin{figure}[htbp]
   \centering
   \includegraphics*[width=100mm]{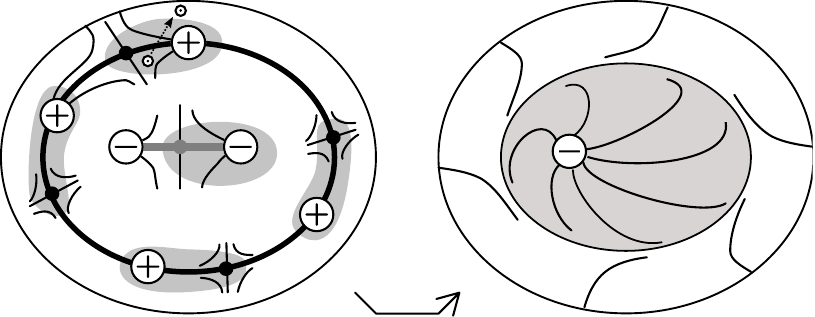} % requires the graphicx package
   \caption{
(Left) From a transverse overtwisted disk to a usual overtwisted disk. 
The graphs $G_{++}$ and $G_{--}$ are depicted by black and gray bold lines, respectively. 
A dot $\odot$ represents an intersection of $K$ and $D_{\sf trans}$ which is moved away from the gray regions before applying the Giroux elimination lemma to the gray regions.
(Right) Disk $D'$ and an overtwisted disk $D'$ (highlighted in gray).}
\label{fig:transtousual}
\end{figure}

The set of transverse links up to transverse isotopy is naturally identified, through positive transverse push-off, with the set of Legendrian links up to Legendrian isotopy and negative stabilization \cite{efm,eh}.

The proof of \cite[Theorem 4.1.4]{bo} shows that every positive intersection of a Legendrian link and an overtwisted disk can be removed by a negative stabilization of the Legendrian link.

Therefore each positive intersection of $K_0$ and the overtwisted disk $D$ can be removed by a suitable transverse isotopy. 
That is, there exists a link $K_1$ that is transversely isotopic to $K_0$ such that $\#(K_1\cap D) = \#^-(K_1 \cap D) = \#^-(K_0\cap D)$. 
We conclude
\[ d(K) \leq \#(K_1\cap D) = \#^-(K_0 \cap D) \leq \#^-(K_0 \cap D') = \#^-(K_0 \cap D_{\sf trans}) = d_{\sf trans}(K).\]

Next we show that $d(K) \geq d_{\sf trans}(K)$. 
Let $D$ be an overtwisted disc in $(M, \xi)$ that intersects $K$ at $d(K)$ points.

Take a slightly larger disc, $D'$, which contains $D$ in its interior and is bounded by a positive transverse push-off of the Legendrian unknot $\partial D$ so that $D' \cap K = D\cap K$. 

Using transverse isotopy we make $K$ disjoint from the binding of the open book.
Following Pavalescu's proof of Alexander theorem \cite[Theorem 3.2]{pav} one can find an isotopy of $M$ preserving each page of the open book set-wise and taking the non-braided part of $\partial D' \cup K$ (subsets which are not positively transverse to pages) into a neighborhood of the binding.

Inside the neighborhood of the binding we make $\partial D' \cup K$ braided  with respect to the open book using \cite{Ben}. 
We call the resulting link and disk $K'$ and $D''$, respectively. 
It is possible that new positive intersection points of $D''$ and $K'$ may be created if a component of $K$ is transversely isotopic to a binding component.
However no new negative intersection points will be introduced. 
Hence $\#^{-}(K'\cap D'') = \#^{-}(K \cap D') \leq d(K)$.

Fixing $\partial D''$ and $K'$ and following the proof of \cite[Theorem 3.3]{ik2} we perturb $D''$ so that the resulting disk, $D'''$, admits an essential open book foliation. 
This process can be done without introducing new intersection points with $K'$ hence $\#^{-}(K' \cap D''')=\#^{-}(K'\cap D'')$.

Since the Bennequin-Eliashberg inequality does not hold 
$$
{\rm sl}(\partial D''', [D''']) = {\rm sl}(\partial D'', [D'']) = {\rm sl}(\partial D', [D']) = {\rm tb}(\partial D, [D]) - {\rm rot}(\partial D, [D]) = 1 \nleq -\chi(D''')
$$ 
we can apply the proof of \cite[Theorem 4.3]{ik1-1} to $D'''$ and obtain a transverse overtwisted disc, $D_{\sf trans}$. 
By the nature of this construction we have 
\begin{eqnarray}
\#^-(K' \cap D_{\sf trans}) &=& \#^-(K' \cap D''') \nonumber  \\
\#^+(K' \cap D_{\sf trans}) &\geq& \#^+(K' \cap D''') \label{+}
\end{eqnarray}
where a strict inequality `$>$' in (\ref{+}) may hold only when a component of $K' $ is transversely isotopic to  a binding component. 
Summing up, we have
$$d_{\sf trans}(K) \leq \#^-(K'\cap D_{\sf trans}) = \#^-(K' \cap D''') = \#^{-}(K'\cap D'') =  \#^{-}(K \cap D')\leq d(K).$$
\end{proof}

Many properties of quasi-right-veering are studied Sections \ref{subsection:QRV-RV} and \ref{subsection:transv-links}. 
The following Theorem~\ref{theorem:depthone}  gives another property of quasi-right-veering.  
One may also apply Theorem~\ref{theorem:depthone} to the study of knots and links of large depth.

\begin{definition}[Axis-augmented transverse link for a closed braid]
Let $(S,\phi)$ be an open book decomposition of a contact 3-manifold $(M,\xi)$ and $\cL=[((S,\varphi),L)]$ be a closed braid with respect to an open book $(S,\phi)$. 
The \emph{axis-augmented transverse link for $\cL$} is a transverse link represented by $B\cup L$, where $B$ denotes the binding of an abstract open book $(S,\vphi)$ supporting $(M,\xi)$. 
\end{definition}

\begin{lemma}
The axis-augmented transverse link for a closed braid $\cL$ is well-defined up to contactomorphism.
\end{lemma}

\begin{proof}
Suppose that $\vphi$ and $\vphi' \in \Diff(S, \partial S)$ are isotopic. 
Fix a contact structure $\xi_{(S, \vphi)}$ on $M_{(S, \vphi)}$ (resp. $\xi_{(S, \vphi')}$ on $M_{(S, \vphi')}$) that is supported by $(S, \vphi)$ (resp. $(S, \vphi')$).  
Starting with an isotopy between $\vphi$ and $\vphi'$ 
we have a contactomorphism $\varrho 
: (M_{(S,\vphi)},\xi_{(S,\vphi)}) \rightarrow (M_{(S,\vphi')},\xi_{(S,\vphi')})$ as constructed in (\ref{eq:vrho}). 
Let $B$ (resp. $B'$) be the binding of the open book decomposition on $M_{(S,\vphi)}$ (resp. $M_{(S,\vphi')}$). 
When closed braids $((S,\vphi),L)$ and $((S,\vphi'),L')$ are equivalent, the link 
$\varrho(B\cup L)$ is transversely isotopic to $B' \cup L'$. Thus, up to a choice of identification $(M_{(S,\vphi)},\xi_{(S,\vphi)}) \cong (M,\xi)$, the transverse link type of $B \cup L$ is uniquely determined.
\end{proof}

%Thus, $d(K)=d(B\cup L)$ is an invariant of $\cL$. 

%\begin{theorem}
%\label{theorem:depthone}
%Let $(S, \vphi)$ be an abstract open book supporting $(M, \xi)$. 
%Let $B$ denote the binding of the induced open book decomposition  of $M$ and $L$ be a closed braid with respect to $(S, \vphi)$. 
%Let $K := B \cup L$ which is a transverse link in $(M, \xi)$.
%The depth $d(K)=1$ if and only if the braid $\cL=[(S, \vphi), L)]$ is not quasi-right-veering. 
%\end{theorem}

\begin{theorem}
\label{theorem:depthone}
Let $\cL$ be a closed braid in open book $(S,\phi)$.
The depth of the axis-augmented transverse link for $\cL$ is one if and only if the braid $\cL$ is not quasi-right-veering. 
\end{theorem}

When the closed braid $L$ is empty we can reprove the following:

\begin{corollary}\label{cor:d(B)}
 \cite[Corollary 1]{ik6}
The depth $d(B) = 1$ if and only if $\phi = [\vphi]$ is not right-veering.   
\end{corollary}

\begin{proof}[Proof of Theorem~\ref{theorem:depthone}]

In the following, we take an abstract open book $(S,\vphi)$ and a  closed braid $L$ so that $\cL=[((S,\vphi),L)]$, and let $K=L\cup B$ where $B$ denotes the binding of the open book decomposition on $M_{(S,\vphi)}$.

($\Leftarrow$) 
Suppose that the braid $\cL$ is not quasi-right-veering.
As in the proof of Theorem~\ref{theorem:main}, we can construct a transverse overtwisted disk with only one negative elliptic point in the complement of $L$. 
By Theorem~\ref{theorem:depth} we have $d(K) \leq 1$. 
On the other hand, since the binding of any open book is non-loose \cite{ev} and $K$ contains the binding $B$ we have $d(K)\geq d(B)\geq 1$.  

($\Rightarrow$) 
Assume that $d(K)=1$. 
Let $D$ be an overtwisted disk in $(M, \xi)$ satisfying $\#(K\cap D)=d(K)=1$. 
Since the complement
of the binding of a supporting open book decomposition is tight \cite{ev}, $\#(D\cap K) = \#(D\cap B)=1$ and $\#(D\cap L)=0$. 

Following the proof of Theorem ~\ref{theorem:depth} (the second half showing $d(K) \geq d_{\sf trans}(K)$) we can construct starting from $D$ a transverse overtwisted disk $D_{\sf trans}$ in the complement of $L$ such that $$\#^-(K \cap D_{\sf trans}) = \#^-(B\cap D_{\sf trans}) =1.$$

Let $v \in B \cap D_{\sf trans}$ denote the unique negative intersection point. 
That is, $v$ is the unique negative elliptic point in the open book foliation $\F(D_{\sf trans})$ of $D_{\sf trans}$. 
Assume that $v$ lies on a boundary component $C$ of $S$.
For a regular page $S_t$ of the open book let $b_t \subset S_t$ denote the unique b-arc in $\F(D_{\sf trans})$ that ends at $v$. 
We use $v$ as the base point $\ast_{C}$ of $C$. 
Recall the projection map (\ref{eq:projection-p}) $$p: M_{(S, \vphi)}  \rightarrow S.$$  
We view the image $p(b_t)$ as an element of $\mathcal{A}_{C}(S,P)$ where $P= p(L \cap S_0)$ is a set of punctures given by the intersection of the braid $L$ and the page $S_0$. 
%The equations (\ref{eq:conjugation}) in the proof of Proposition~\ref{prop:isotopic-braids} show that it is not necessary for the rest of the argument that $P$ is included in a neighborhood of some boundary component. 

Let $S_{t_1}, \ldots, S_{t_k}$ $(0<t_1<\cdots < t_k <1)$ be the singular pages of the open book foliation $\F(D_{\sf trans})$  and  $\varepsilon>0$ be a sufficiently small number such that  $S_{t_i}$ is the only singular page in the interval $(t_i - \e, t_i+\e)$.
Since $D_{\sf trans}$ is a transverse overtwisted disk with one negative elliptic point, by the definition of a transverse overtwisted disk \cite[Definition 4.1]{ik1-1}, all the hyperbolic points of $\F(D_{\sf trans})$ are positive.  
This shows that $p(b_{t_{i}-\varepsilon}) \prec_{\sf right} p(b_{t_{i}+\varepsilon})$ with $\Int(p(b_{t_{i}-\varepsilon})) \cap \Int(p(b_{t_{i}+\varepsilon}))= \emptyset$ for all $i=1,\dots,k$ (see Figure \ref{fig:key} (ii), or consult Observation 2.5 of \cite{ik4}). 
Let us put $$\gamma_i:= p(b_{t_{i} + \varepsilon}) = p( b_{t_{i+1}-\varepsilon}) \in \A.$$ 
Then the sequence of arcs satisfies 
\[ [\vphi_{L}](p(b_1))= p(b_0) = \gamma_0 \prec_{\ri} \gamma_1 \prec_{\ri} \cdots \prec_{\ri} \gamma_{k} = p(b_{1})\]
and $\Int(\gamma_{i}) \cap \Int( \gamma_{i+1}) = \emptyset$;
hence, $[\vphi_L] \in \MCG(S,P)$ is not quasi-right-veering. 
\begin{figure}[htbp]
\begin{center}
\includegraphics*[%bb= 120 530 484 730,
width=110mm]{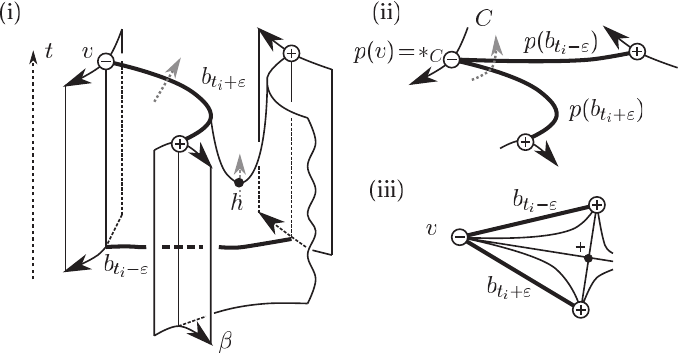}
\caption{(i): A positive hyperbolic point $h$ (saddle tangency). The gray dashed arrow indicate positive normal vectors to the surface. 
Black arrows indicate the orientations of the binding components. 
(ii) Comparison of the b-arcs $p(b_{t_i -\varepsilon})$ and $p(b_{t_i + \varepsilon})$ projected to $S$. 
(iii) Corresponding portion in the open book foliation $\F(D_{\sf trans})$. 
}
\label{fig:key}
\end{center}
\end{figure}
\end{proof}

\section{Very positive FDTC and non-loose links}\label{sec:non-loose} 

Proposition \ref{prop:FDTCvsqrv} and Theorem~\ref{theorem:main} show that negative FDTC $c(\phi, \cL, C)< 0$ does not always imply looseness of the closed braid $\cL$. 
This makes a sharp contrast to the empty braid case, where the negative FDTC $c(\phi, C)< 0$ implies that the contact structure $\xi_{(S, \phi)}$ is overtwisted. 

On the other hand, if the FDTC is very positive then there is some similarity between non-empty braid case and empty braid case. 
In \cite[Corollary 1.2]{ik4} it is proved that a planar open book $(S,\phi)$ with $c(\phi,C)>1$ for every boundary component $C$ supports a tight contact structure. 
We may regard this as a special case ($\cL=\emptyset$) of the following theorem.

\begin{theorem}
\label{theorem:nonloose}
Let $(S,\phi)$ be a planar open book decomposition of a contact 3-manifold $(M,\xi)$. If a transverse link $\mathcal{T} \subset (M,\xi)$ is represented by a closed braid $\cL=[(S,\vphi),L]$ such that $c(\phi,\cL,C)>1$ for every boundary component $C$ of $S$, then $\mathcal{T}$ is non-loose.
\end{theorem}

\begin{proof}

Let $(S,\vphi)$ be an abstract open book that supports $(M, \xi)$, such that $P:=p(L\cap S_0) \subset \nu(\partial S)$ and $\vphi=id$ on $\nu(\partial S)$, and let $L$ be a closed braid with respect to $(S,\vphi)$ that represents $\cL$.
By (\ref{eqML}) we have
$$\left( (S \setminus P) \times [0,1] \right)/ 
\sim_{\varphi_L} \  \simeq \  %M
M_{(S,\vphi)}\setminus L.$$  
%\marginpar{\tiny 3 sentences used to be here were removed}
%Recall the forgetful map $f: \MCG(S,P)\to \MCG(S)$ in the generalized Birman exact sequence (\ref{eqn:Birman}). 
%Note that $f([\vphi_L])=\phi \in \MCG(S)$. 
%In the following argument, we may use the abstract open book $(S, P, \vphi_L)$} instead of $(S, \vphi)$. 

Assume that $L$ is loose. 
By Theorem~\ref{theorem:depth} there exists a transverse overtwisted disk $D$ in $M_{(S,\vphi)} \setminus L$.  %\marginpar{\tiny Thm5.2 only says $d^-_{\sf trans}(K)=0$. It doesn't say $d^+_{\sf trans}(K)=0$} 
Applying the proof of \cite[Theorem 1.1]{ik4} to the diffeomorphism $\vphi_L\in\Diff(S, P, \partial S)$, we can construct a transverse overtwisted disk $D'$ in $M_{(S,\vphi)} \setminus L$ such that every b-arc of $\F(D')$ ending at a valence $\leq 1$ vertex of the graph $G_{--}(D')$ is projected to an essential arc in the punctured page $S \setminus P$ under the map $p:M_{(S, \vphi)}\to S$ in (\ref{eq:projection-p}). 
Using \cite[Lemma 5.11]{ik2} the existence of such a disk $D'$ implies that $c(\phi, \cL, C)=c([\vphi_L], C) \leq 1$ for some boundary component $C$ of $S $.  
\end{proof}    

\section{Comparison of three definitions of right-veering}
\label{sec:comparison}

In this section we discuss comparison of several proposed  definitions of right-veering-ness for mapping classes in $\MCG(S, P)$.

\begin{definition}
We say that an arc $\gamma:[0,1] \rightarrow S$ is $\partial$-$P$ (resp. $\partial$-$\partial$) arc if the following are all satisfied: 
\begin{enumerate}
\item 
$\gamma(0) \in \partial S$ and $\gamma$ is transverse to $\partial S$ at $\gamma(0)$.
\item 
$\gamma(t) \in \Int(S) \setminus P$ for $t\in (0,1)$.
\item 
$\gamma(1) \in P$ (resp.  $\gamma(1) \in \partial S$ and $\gamma$ is transverse to $\partial S$ at $\gamma(1)$).
\item 
$\Int(\gamma)$ is embedded in $S\setminus P$ and not boundary-parallel.
\end{enumerate}
For a boundary component $C$ of $S$, we say that a $\partial$-$P$ or $\partial$-$\partial$ arc is \emph{based on $C$} if $\gamma(0) \in C$.
\end{definition}

As natural generalizations of the right-veering-ness for $\phi\in\MCG(S)$ to $\psi\in\MCG(S,P)$ there are three candidates. 
\begin{definition}
\label{def:dd-right-veering}
For a boundary component $C$ of $S$ we say that $\psi \in \MCG(S,P)$ is
\begin{enumerate}
\item 
$\partial$-$(\partial+P)$ {\em right-veering with respect to $C$} if $\gamma \prec_{\ri} \psi(\gamma)$ or $\gamma = \psi(\gamma)$ for all $\partial$-$\partial$ and $\partial$-$P$ arcs $\gamma$ based on $C$.
\item 
$\partial$-$\partial$ {\em right-veering with respect to $C$} if $\gamma \prec_{\ri} \psi(\gamma)$ or $\gamma = \psi(\gamma)$ for all $\partial$-$\partial$ arcs $\gamma$ based on $C$.
\item 
$\partial$-$P$ {\em right-veering with respect to $C$} if $\gamma \prec_{\ri} \psi(\gamma)$ or $\gamma = \psi(\gamma)$ for all $\partial$-$P$ arcs $\gamma$ based on $C$.

\item
We say that $\psi \in \MCG(S,P)$ is $\partial$-$(\partial+P)$, $\partial$-$\partial$, or, $\partial$-$P$ {\em right-veering}, respectively, if $\psi$ is $\partial$-$(\partial+P)$, $\partial$-$\partial$, or, $\partial$-$P$ right-veering, respectively, with respect to every boundary component of $S$.
\end{enumerate}
\end{definition}

The $\partial$-$\partial$ right-veering is used by Baldwin, Vela-Vick and V\'ertesi in \cite{bvv}.
It is easy to see that our Definition~\ref{defn:right-veering}  of  right-veering is equivalent to the $\partial$-$\partial$ right-veering. 
Recall that in Definition~\ref{defn:right-veering} we only consider $\partial$-$\partial$ arcs starting from the distinguished base point $\ast_C \in C$. This restriction is just to define the orderings $\prec_{\ri}$ and $\ll_{\ri}$ on $\A$.

On the other hand, Baldwin and Grigsby \cite{bg} and Plamenevskaya \cite{pl}  use the notion of $\partial$-$P$ right-veering to study the classical braid group $\MCG(D^{2},P)$.

Baldwin and Grigsby ask in \cite[Remark 3.3]{bg} whether these two superficially different notions of ``right-veering'' are equivalent or not.

The following example shows that the notions (2) and (3) of ``right-veering with respect to $C$'' are in general not exactly the same: 

\begin{example}
\label{example:diff}
Assume that $S$ has more than one boundary component and non-empty marked points $P \subset \Int(S)$. 
Let $C$ and $C'$ be distinct boundary components. 
Clearly $T_{C'}^{-1} \in \MCG(S,P)$ is not $\partial$-$\partial$ right-veering with respect to $C$. 
On the other hand $T_{C'}^{-1}$ preserves all $\partial$-$P$ arcs based on $C$. This means that $T_{C'}^{-1}$ is $\partial$-$P$ right-veering with respect to $C$.

More generally, let $\psi \in \MCG(S,P)$ be a $\partial$-$P$ right-veering map with respect to $C$, and 
suppose that $\psi(\gamma)=\gamma$ for some $\partial$-$\partial$ arc $\gamma$ connecting distinct $C$ and $C'$. 
Then  $T_{C'}^{-1}\psi$ is still $\partial$-$P$ right-veering with respect to $C$, but is not  $\partial$-$\partial$ right-veering with respect to $C$ since $T_{C'}^{-1}\psi(\gamma) = T_{C'}^{-1}(\gamma) \prec_{\ri} \gamma$.
\end{example}

It turns out that the difference between $\partial$-$\partial$ right-veering and $\partial$-$P$ right-veering only shows up when $\psi\in \MCG(S,P)$ involves negative Dehn twists along boundary components like in Example \ref{example:diff}. 

\begin{definition}\label{def:special}
We say that $\psi \in \MCG(S,P)$ is {\em special} with respect to $C$ if the following two conditions are satisfied. 
\begin{itemize}
\item
$\psi$ is not $\partial$-$\partial$ right-veering with respect to $C$. 
\item
If a $\partial$-$\partial$ arc $\gamma$ that is based on $C$ and ending at $C'$ has $\psi(\gamma) \prec_\ri \gamma$ then $C' \neq C$ and $\psi(\gamma)= T_{C'}^{-n}(\gamma)$ for some $n>0$.
\end{itemize}
That is, a special map $\psi$ is not $\partial$-$\partial$ right-veering with respect to $C$ only because of negative Dehn twists about some other boundary component $C'$.
\end{definition}

\begin{theorem}
\label{theorem:equivalencerv}
Let $\psi \in \MCG(S,P)$. 
\begin{enumerate}
\item If $\psi$ is $\partial$-$\partial$ right-veering with respect to $C$, then $\psi$ is $\partial$-$P$ right-veering with respect to $C$.
\item If $\psi$ is $\partial$-$P$ right-veering with respect to $C$ then either 
\begin{itemize}
\item
$\psi$ is $\partial$-$\partial$ right-veering with respect to $C$, or, 
\item
$\psi$ is special with respect to $C$.
\end{itemize}
\end{enumerate} 
\end{theorem}

\begin{proof}
We prove both (1) and (2) by showing the contrapositives. 

First we prove (1). Assume to the contrary that there is a $\partial$-$P$ arc $\gamma$ based on $C$ with $\psi(\gamma) \prec_{\sf right} \gamma$. Let $\kappa \in\A$ be a properly embedded arc which is the boundary of a regular neighborhood of $\gamma$ in $S$. Then we see that $\kappa$ is a $\partial$-$\partial$ arc with $\psi(\kappa) \prec_{\sf right} \kappa$.

To see (2), assume to the contrary that $\psi$ is not $\partial$-$\partial$ right-veering with respect to $C$ and is not special with respect to $C$.  
Then there exists a $\partial$-$\partial$ arc $\gamma$ based on $C$ such that $\psi(\gamma) \prec_{\sf right} \gamma$.
We put $\psi(\gamma)$ and $\gamma$ so that they intersect efficiently.
Our goal is to show that there exits a $\partial$-$P$ arc $\kappa$ based on $C$ with $\kappa(0)=\gamma(0)$ and $$\psi(\gamma) \prec_{\sf right} \kappa \prec_{\sf right} \gamma.$$ % if $\gamma$ is not ``bad'', which is defined in Definition~\ref{def:bad} below. 
This shows $\psi(\kappa) \prec_{\ri} \psi(\gamma) \prec_{\ri} \kappa$; hence, $\psi$ cannot be $\partial$-$P$ right-veering with respect to $C$.

If $\#(\gamma, \psi(\gamma))=m>0$, we put $\Int(\gamma) \cap \Int(\psi(\gamma))=\{p_1,\ldots,p_m\} =\{q_1,\ldots,q_m\}$, where $p_{i}=\gamma(t_i)$ with $0 < t_1 < t_2 < \cdots < t_m < 1$ and $q_{i}=(\psi(\gamma))(s_i)$ with $0 < s_1 < s_2 < \cdots < s_m < 1$.
If $\#(\gamma, \psi(\gamma))=m=0$ we put $t_1=s_1=1$ and $p_{1}=q_{1}=\gamma(1)$.

Suppose that $q_1 = p_{k}$.
Let 
$$\delta:= \gamma|_{[0,t_k]} \ast (\psi(\gamma)|_{[0,s_1]})^{-1}$$ then $\delta$ is an oriented simple closed curve in $S \setminus P$. 
Here $\ast$ denotes concatenation of paths read from left to right, and $(-)^{-1}$ means the arc with reversed orientation. 
If $\delta$ is separating, we denote by $R$ the connected component of $S \setminus (\delta \cup P)$ that lies on the left side of $\delta$ with respect to the orientation of $\delta$. If $\delta$ is non-separating let $R:=S \setminus (\delta \cup P)$.

\begin{definition}\label{def:bad}
We say that the arc $\gamma$ is \emph{bad} if the following two properties are satisfied:
\begin{itemize}
\item $R$ is an annulus (possibly a pinched annulus if $m=0$) with no punctures. (In particular, $\delta$ is separating.) 
\item The sign of the intersection of $\gamma$ and $\psi(\gamma)$ (in this order) at $q_1$ is positive. 
\end{itemize}
\end{definition}

Assume that $\gamma$ is bad. Let $C' = \partial R \setminus \delta$. Note that $C'$ is a boundary component of $S$. 
Since $\gamma$ and $\psi(\gamma)$ intersect efficiently and $\delta$ is separating, we can see that $\psi(\gamma)$ cannot exit out of the annulus $R$ and $C\not=C'$. See Figure~\ref{fig:badarc}.
Therefore we have:

\begin{claim}
\label{claim:bad}
If $\gamma$ is bad then $C' \neq C$ and $\psi(\gamma) = T_{C'}^{-n}(\gamma)$ for some  $n>0$.
\end{claim}

\begin{figure}[htbp]
\begin{center}
\includegraphics*[bb= 185 615 441 706,width=90mm]{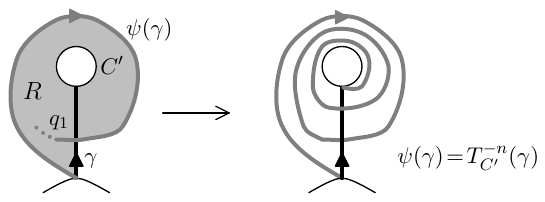}
\caption{A bad arc $\gamma$  and its image $\psi(\gamma)$.}
\label{fig:badarc}
\end{center}
\end{figure}

Since we assume that $\psi$ is 
not $\partial$-$\partial$ right-veering with respect to $C$ and is not special with respect to $C$, Claim~\ref{claim:bad} implies that $\gamma$ is not bad.

Knowing that $\gamma$ is not bad, we consider two cases to construct $\kappa$:

\noindent
\underline{\bf Case 1}: $R$ is an annulus with punctures or a non-annulus surface with or without punctures.

%\noindent
The sign of the intersection of $\gamma$ and $\psi(\gamma)$ at $q_1$ can be either positive or negative. 
Take an arc $\gamma'$ in $S \setminus (P \cup \gamma \cup \delta)$ which connects $q_1$ and some puncture point and efficiently intersects $\psi(\gamma)|_{[s_1, 1]}$.
 
\noindent
\underline{\it Case 1A}:  
There exists such an arc $\gamma'$ which  lies on the left side of $\gamma$ near $q_1$. 

%\noindent
In this case, define  $\kappa := \gamma|_{[0,t_k]}\ast \gamma'$.

\noindent
\underline{\it Case 1B}: 
No such arc can exist on the left side of $\gamma$ near $q_1$, so  $\gamma'$ lies on the right side of $\gamma$ near $q_1$.

%\noindent
%If $\delta$ is separating and $R$ is a punctured disk or a punctured annulus
If $R$ contains punctures 
then let %$\kappa \subset (R \setminus (R\cap \gamma))$ 
$\kappa\subset R$ be an arc connecting $\gamma(0)$ and one of the punctures in $R$ and satisfying $\psi(\gamma) \prec_\ri \kappa \prec_\ri \gamma$ and $
\Int(\kappa) \cap\delta=\emptyset$. (We do not use $\gamma'$ here.)

Now we may assume that $R$ is a %planar surface with more than two boundary components or a surface with genus $\geq 1$. 
non-annular surface with no punctures.
We can take an arc $\gamma''$ in $R \setminus (R \cap \gamma')$ such that:
\begin{itemize}
\item 
$\gamma''(0)=\gamma''(1)=\gamma(0)$.
\item 
$\psi(\gamma) \prec_{\ri} \gamma'' \prec_{\ri} \gamma$.
\item
$\Int(\gamma'') \cap \delta = \emptyset$.
\item 
$\gamma''$ is not parallel to $\delta$.
\item
$\gamma''$ and $\gamma$ efficiently intersect.
\end{itemize}
Let 
$q'':=\gamma''(u) = \gamma(t) \in \gamma''\cap\gamma$ be the intersection point such that $\gamma''|_{(0,u)}$ is disjoint from $\gamma$. If $\Int(\gamma'') \cap \Int(\gamma)=\emptyset$ then we take $q'':=\gamma''(1)=\gamma(0)$. 
Namely, $u=1$ and $t=0$. 
Define:
$$
\kappa:= \left\{
\begin{array}{ll}
\gamma''|_{[0,u]} \ast \gamma|_{[t, t_k]} \ast \gamma' & (\mbox{if } t<t_k) \\
\gamma''|_{[0,u]} \ast \gamma|_{[t_k, t]} \ast \gamma' & (\mbox{if } t_k<t)
\end{array}
\right.
$$

\begin{figure}[htbp]
\begin{center}
\includegraphics*[bb= 139 562 471 657, width=110mm]{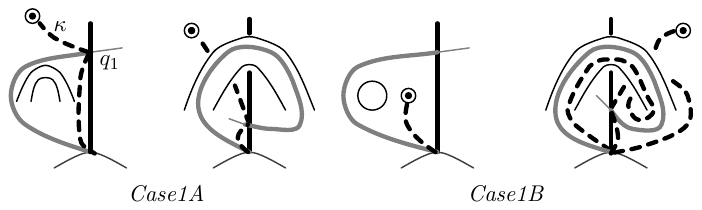}
\caption{{\bf Case 1}. A $\partial$-$P$ arc $\kappa$ (dashed arc) is chosen so that it does not intersect $\gamma$ (black bold line) and $\psi(\gamma)|_{[0,s_1]}$ (gray bold arc), possibly with one exceptional point $q_1$.}
\label{fig:case1}
\end{center}
\end{figure}

\noindent
\underline{\bf Case 2}: $R$ is an annulus with no punctures, and the sign of the intersection of $\gamma$ and $\psi(\gamma)$ at $q_1$ is negative.

\noindent
%Since $\gamma$ is not bad the sign of the intersection of $\gamma$ and $\psi(\gamma)$ at $q_1$ is negative. 
Let $k' (\neq k)$ be the number satisfying $q_2=p_{k'}$.

\noindent
\underline{\it Case 2A}: $k' < k$. 

\noindent
Since $\delta$ is separating the sign of the intersection of $\gamma$ and $\psi(\gamma)$ at $q_2$ is positive. 
Take an arc $\gamma'$ in $S \setminus (P \cup \gamma \cup \psi(\gamma)|_{[0,s_2]})$ which connects $q_2$ and a puncture point and efficiently intersects $\psi(\gamma)|_{[s_2,1]}$.
Then put $\kappa:= \psi(\gamma)|_{[0,s_1]} \ast (\gamma|_{[t_{k'},t_{k}]})^{-1} \ast \gamma'$.

\noindent
\underline{\it Case 2B}: $k< k'$. 

\noindent
Let $\gamma'$ be an arc in $S \setminus (P \cup \gamma \cup \psi(\gamma)|_{[0,s_2]})$ that connects $\gamma(0)$ and a puncture point.
Put 
\[
\kappa := \begin{cases}
\gamma|_{[0,t_{k'}]} \ast (\psi(\gamma)|_{[s_{1},s_{2}]})^{-1} \ast (\gamma|_{[0,s_{1}]})^{-1} \ast \gamma' & ( \text{ if } \gamma \prec_{\ri} \gamma'
)\\
\gamma|_{[0,t_{k'}]} \ast (\psi(\gamma)|_{[s_{1},s_{2}]})^{-1} \ast C \ast (\gamma|_{[0,s_{1}]})^{-1} \ast \gamma' &  ( \text{ if } \gamma' \prec_{\ri} \gamma)
\end{cases}
\]
In the second case in order to make $\kappa$ embedded it turns along $C$. 
\begin{figure}[htbp]
\begin{center}
\includegraphics*[bb= 140 610 482 708,width=110mm]{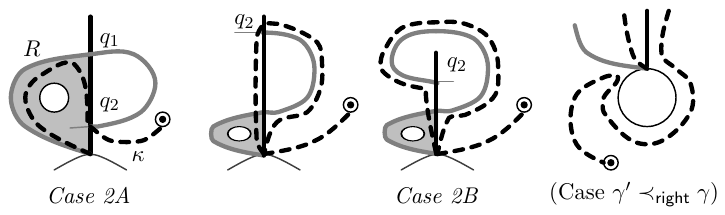}
\caption{{\bf Case 2}.
Construction of a $\partial$-$P$ arc $\kappa$ (dashed).
$\kappa$ does not intersect $\gamma$ (black bold line) and $\psi(\gamma)|_{[0,s_2]}$ (gray bold arc), possibly with exceptions near $q_1$, $q_2$ and $\gamma(1)$ (if $\gamma(1)\in C$).}
\label{fig:case2}
\end{center}
\end{figure}
\end{proof}

As a consequence of Theorem~\ref{theorem:equivalencerv}, the three notions of right-veering with respect to \emph{all} the boundary components, which is a condition closely related to tight contact structures, are equivalent. 
In particular, if $S$ has connected boundary then the three notions are equivalent. 

\begin{corollary}
\label{cor:rveer}
For $\psi \in \MCG(S,P)$ the following are equivalent.
\begin{enumerate}
\item $\psi$ is $\partial$-$(\partial+P)$ right-veering.
\item $\psi$ is $\partial$-$\partial$ right-veering.
\item $\psi$ is $\partial$-$P$ right-veering.
\end{enumerate}
\end{corollary}

Therefore in the case of the braid group $B_{n}=\MCG(D^2, \{n\text{ points}\})$ the proposed definitions of right-veering in \cite{bvv} and \cite{bg,pl} are the same.
Also, we remark that the subtle difference between $\partial$-$P$ right-veering with respect to $C$ and $\partial$-$\partial$ right-veering with respect to $C$   (existence of a special mapping class $\psi$) only occurs when $c(\psi,C)=0$.

\begin{remark}
One may come up with still different candidates of right-veering. Instead of using embedded arcs, one may use immersed arcs. 
However, one can check that immersed $\partial$-($\partial+P$) (resp. $\partial$-$\partial$, $\partial$-$P$) right-veering with respect to $C$ is equivalent to the (embedded) $\partial$-($\partial+P$) (resp. $\partial$-$\partial$, $\partial$-$P$) right-veering with respect to $C$.
\end{remark}

\section*{Acknowledgements}
The authors would like to thank John Etnyre for many useful comments. % including one on Proposition~\ref{prop:covering}. 
TI was partially supported by JSPS KAKENHI grant number 15K17540 and 16H02145.
KK was partially supported by NSF grant DMS-1206770 and Simons Foundation Collaboration Grants for Mathematicians.

\end{document}